\documentclass[10pt]{article}

\usepackage{amsfonts}
\usepackage{mathrsfs}
\usepackage{amsmath}
\usepackage{amsthm}
\usepackage{amssymb}
\usepackage{latexsym,enumitem}
\usepackage[all]{xy}

\usepackage[square,numbers,sort]{natbib}

\usepackage{hyperref}
\hypersetup{%
  pdfborder = {0 0 0},
  colorlinks,
  citecolor=red!50!black,
  filecolor=Darkgreen,
  linkcolor=blue!40!black,
  urlcolor=cyan!50!black!90
}

\numberwithin{equation}{section}

\newtheorem{lemma}{Lemma}[section]

\newtheorem{thrm}{Theorem}[section]
\newtheorem{prop}{Proposition}[section]
\newtheorem{defn}{Definition}[section]
\newtheorem{rmk}{Remark}[section]
\newtheorem{crl}{Corollary}[section]

\newcommand{\nc}{\newcommand}

\nc{\al}{\alpha}
\nc{\eps}{\epsilon}
\nc{\ga}{\gamma}
\nc{\Ga}{\Gamma}
\nc{\ka}{\kappa}
\nc{\la}{\lambda}
\nc{\om}{\omega}
\nc{\si}{\sigma}
\nc{\Si}{\Sigma}
\nc{\bsi}{\boldsymbol\sigma}
\nc{\bSi}{\boldsymbol\Sigma}
\nc{\Ups}{\upsilon}
\nc{\vphi}{\varphi}

\nc{\btau}{\boldsymbol\tau}
\nc{\bdel}{\boldsymbol\delta}

\nc{\id}{\mathrm{id}}
\nc{\gr}{\mathrm{gr}}

\nc{\Ug}{U\mathfrak{g}}
\nc{\Ub}{U\mathfrak{b}}

\nc{\Hk}{\mathsf{H}}
\nc{\ombH}{\overline{\mathbf{H}}}

\nc{\ud}{\underline}
\nc{\tl}{\tilde}

\nc{\mbA}{\mathbf{A}}
\nc{\mbb}{\mathbf{b}}
\nc{\mbB}{\mathbf{B}}
\nc{\mbc}{\mathbf{c}}
\nc{\mbC}{\mathbf{C}}
\nc{\mbd}{\mathbf{d}}
\nc{\mbD}{\mathbf{D}}
\nc{\mbe}{\mathbf{e}}
\nc{\mbE}{\mathbf{E}}
\nc{\mbf}{\mathbf{f}}
\nc{\mbF}{\mathbf{F}}
\nc{\mbg}{\mathbf{g}}
\nc{\mbH}{\mathbf{H}}
\nc{\mbh}{\mathbf{h}}
\nc{\mbi}{\mathbf{i}}
\nc{\mbI}{\mathbf{I}}
\nc{\mbj}{\mathbf{j}}
\nc{\mbJ}{\mathbf{J}}
\nc{\mbk}{\mathbf{k}}
\nc{\mbK}{\mathbf{K}}
\nc{\mbL}{\mathbf{L}}
\nc{\mbM}{\mathbf{M}}
\nc{\mbQ}{\mathbf{Q}}
\nc{\mbq}{\mathbf{q}}
\nc{\mbr}{\mathbf{r}}
\nc{\mbT}{\mathbf{T}}
\nc{\mbu}{\mathbf{u}}
\nc{\mbU}{\mathbf{U}}
\nc{\mbv}{\mathbf{v}}
\nc{\mbV}{\mathbf{V}}
\nc{\mbw}{\mathbf{w}}
\nc{\mbW}{\mathbf{W}}
\nc{\mbX}{\mathbf{X}}
\nc{\mbY}{\mathbf{Y}}
\nc{\mbZ}{\mathbf{Z}}

\nc{\mbbA}{\mathbb{A}}
\nc{\mbbB}{\mathbb{B}}
\nc{\mbbD}{\mathbb{D}}
\nc{\mbbF}{\mathbb{F}}
\nc{\mbbV}{\mathbb{V}}
\nc{\mbbH}{\mathbb{H}}
\nc{\mbbK}{\mathbb{K}}
\nc{\mbbL}{\mathbb{L}}
\nc{\mbbP}{\mathbb{P}}
\nc{\mbbU}{\mathbb{U}}

\nc{\mcA}{\mathcal{A}}
\nc{\mcB}{\mathcal{B}}
\nc{\mcC}{\mathcal{C}}
\nc{\mcD}{\mathcal{D}}
\nc{\mcE}{\mathcal{E}}
\nc{\mcF}{\mathcal{F}}
\nc{\mcG}{\mathcal{G}}
\nc{\mcH}{\mathcal{H}}
\nc{\mcK}{\mathcal{K}}
\nc{\mcN}{\mathcal{N}}
\nc{\mcO}{\mathcal{O}}
\nc{\mcQ}{\mathcal{Q}}
\nc{\mcS}{\mathcal{S}}
\nc{\mcP}{\mathcal{P}}
\nc{\mcU}{\mathcal{U}}
\nc{\mcT}{\mathcal{T}}
\nc{\mcV}{\mathcal{V}}
\nc{\mcW}{\mathcal{W}}
\nc{\mcX}{\mathcal{X}}
\nc{\mcY}{\mathcal{Y}}
\nc{\mcZ}{\mathcal{Z}}

\nc{\mfa}{\mathfrak{a}}
\nc{\mfA}{\mathfrak{A}}
\nc{\mfb}{\mathfrak{b}}
\nc{\mfB}{\mathfrak{B}}
\nc{\mfC}{\mathfrak{C}}
\nc{\mfd}{\mathfrak{d}}
\nc{\mfD}{\mathfrak{D}}
\nc{\mfe}{\mathfrak{e}}
\nc{\mfE}{\mathfrak{E}}
\nc{\mff}{\mathfrak{f}}
\nc{\mfF}{\mathfrak{F}}
\nc{\mfg}{\mathfrak{g}}
\nc{\mfgl}{\mathfrak{g}\mathfrak{l}}
\nc{\mfh}{\mathfrak{h}}
\nc{\mfH}{\mathfrak{H}}
\nc{\mfJ}{\mathfrak{J}}
\nc{\mfk}{\mathfrak{k}}
\nc{\mfK}{\mathfrak{K}}
\nc{\mfl}{\mathfrak{l}}
\nc{\mfL}{\mathfrak{L}}
\nc{\mfM}{\mathfrak{M}}
\nc{\mfm}{\mathfrak{m}}
\nc{\mfn}{\mathfrak{n}}
\nc{\mfN}{\mathfrak{N}}
\nc{\mfo}{\mathfrak{o}}
\nc{\mfP}{\mathfrak{P}}
\nc{\mfQ}{\mathfrak{Q}}
\nc{\mfS}{\mathfrak{S}}
\nc{\mfsl}{\mathfrak{s}\mathfrak{l}}
\nc{\mfso}{\mathfrak{s}\mathfrak{o}}
\nc{\mfsp}{\mathfrak{s}\mathfrak{p}}
\nc{\mft}{\mathfrak{t}}
\nc{\mfU}{\mathfrak{U}}
\nc{\mfu}{\mathfrak{u}}
\nc{\mfUqsl}{\mathfrak{U}_q\mathfrak{sl}}
\nc{\mfUsl}{\mathfrak{Usl}}
\nc{\mfV}{\mathfrak{V}}
\nc{\mfX}{\mathfrak{X}}
\nc{\mfY}{\mathfrak{Y}}
\nc{\mfz}{\mathfrak{z}}

\nc{\mrmd}{\mathrm{d}}

\nc{\sal}{\check{\alpha}}
\nc{\cbeta}{\check{\beta}}
\nc{\cd}{\check{d}}
\nc{\cf}{\check{f}}
\nc{\cdelta}{\check{\delta}}
\nc{\ccr}{\check{r}}
\nc{\cs}{\check{s}}

\nc{\bv}{\breve{v}}

\nc{\tc}{\tilde{c}}
\nc{\tr}{\tilde{r}}
\nc{\ts}{\tilde{s}}
\nc{\tv}{\tilde{v}}

\nc{\msA}{\mathsf{A}}
\nc{\msB}{\mathsf{B}}
\nc{\msC}{\mathsf{C}}
\nc{\msc}{\mathsf{c}}
\nc{\msD}{\mathsf{D}}
\nc{\msd}{\mathsf{d}}
\nc{\mse}{\mathsf{e}}
\nc{\msw}{\mathsf{w}}
\nc{\msq}{\mathsf{q}}
\nc{\msg}{\mathsf{g}}
\nc{\msE}{\mathsf{E}}
\nc{\msf}{\mathsf{f}}
\nc{\msF}{\mathsf{F}}
\nc{\msh}{\mathsf{h}}
\nc{\msk}{\mathsf{k}}
\nc{\msH}{\mathsf{H}}
\nc{\msI}{\mathsf{I}}
\nc{\msJ}{\mathsf{J}}
\nc{\msK}{\mathsf{K}}
\nc{\msL}{\mathsf{L}}
\nc{\msP}{\mathsf{P}}
\nc{\msQ}{\mathsf{Q}}
\nc{\msR}{\mathsf{R}}
\nc{\mss}{\mathsf{s}}
\nc{\msS}{\mathsf{S}}
\nc{\msT}{\mathsf{T}}
\nc{\msU}{\mathsf{U}}
\nc{\msv}{\mathsf{v}}
\nc{\msV}{\mathsf{V}}
\nc{\msx}{\mathsf{x}}
\nc{\msX}{\mathsf{X}}
\nc{\msY}{\mathsf{Y}}
\nc{\msZ}{\mathsf{Z}}

\nc{\End}{\mathrm{End}}
\nc{\Ext}{\mathrm{Ext}}
\nc{\Hom}{\mathrm{Hom}}
\nc{\Ima}{\mathrm{Image}}
\nc{\Ind}{\mathrm{Ind}}
\nc{\Ker}{\mathrm{Ker}}
\nc{\RHom}{\mathrm{RHom}}
\nc{\Sym}{\mathrm{Sym}}

\nc{\mtc}{\mathtt{c}}
\nc{\mtD}{\mathtt{D}}
\nc{\mte}{\mathtt{e}}
\nc{\mtE}{\mathtt{E}}
\nc{\mtf}{\mathtt{f}}
\nc{\mtF}{\mathtt{F}}
\nc{\mth}{\mathtt{h}}
\nc{\mtH}{\mathtt{H}}
\nc{\mtV}{\mathtt{V}}
\nc{\mtX}{\mathtt{X}}
\nc{\mty}{\mathtt{y}}

\nc{\ddeg}{\mathtt{deg}}
\nc{\dimm}{\mathtt{dim}}
\nc{\lmod}{\mathtt{lmod}}
\nc{\opp}{\mathtt{opp}}
\nc{\rmod}{\mathtt{rmod}}

\nc{\mmod}{\mathrm{mod}}
\nc{\nbh}{\mathrm{nbh}}

\nc{\lrh}{\leftrightharpoons}
\nc{\iso}{\stackrel{\sim}{\longrightarrow}}
\nc{\liso}{\stackrel{\sim}{\longleftarrow}}
\nc{\wh}{\widehat}
\nc{\wt}{\widetilde}
\nc{\lra}{\longrightarrow}
\nc{\ra}{\rightarrow}
\nc{\into}{\hookrightarrow}
\nc{\onto}{\twoheadrightarrow}

\nc{\mysum}{\textstyle\sum}
\nc{\mysuml}{\textstyle\sum\limits}

\nc{\C}{\mathbb{C}}
\nc{\N}{\mathbb{N}}
\nc{\Z}{\mathbb{Z}}

\nc{\SW}{\mathsf{SW}}

\nc{\ot}{\otimes}
\nc{\op}{\oplus}
\nc{\ol}{\overline}
\nc{\un}{\underline}

\nc{\lan}{\langle}
\nc{\ran}{\rangle}

\nc{\Xp}{x_{j_1,\ldots,j_p}^{-1}}
\nc{\Xeta}{x_{\eta_1,\ldots,\eta_e}^{-1}}

\nc{\eqa}[1]{\begin{align}#1\end{align}}
\nc{\eqn}[1]{\begin{align*}#1\end{align*}}
\nc{\eq}[1]{\begin{equation}#1\end{equation}}

\usepackage{color}
\usepackage[usenames,dvipsnames]{xcolor}
\nc{\red}{\color{red}}
\nc{\blu}{\color{blue}}
\nc{\br}{\color{Brown}}
\nc\el{\nonumber\\}
\nc\nn{\nonumber}

\nc\bb{\mathbb}

\nc\mf[1]{\mathfrak{#1}}

\nc\Tr{{\rm tr}}

\nc{\sm}[1]{\text{\tiny{\rm #1}}}

\nc\tdeg{\stackrel{\sim}{\smash{\deg}\rule{0pt}{1.1ex}}}

\usepackage{sectsty}
\sectionfont{\large}

\usepackage{titletoc}
\usepackage{tocloft}

\titlecontents{section}
 [0em]{}{\thecontentslabel\enspace}{}{\titlerule*[1pc]{.}\contentspage}
\titlecontents{subsection}
 [2em]{\thecontentslabel\enspace}{}{}{\titlerule*[1pc]{.}\contentspage}

\usepackage{etoolbox}
\apptocmd\normalsize{%
 \abovedisplayskip=6pt plus 3pt minus 3pt
 \abovedisplayshortskip=0pt plus 3pt
 \belowdisplayskip=6pt plus 3pt minus 3pt
 \belowdisplayshortskip=6pt plus 3pt minus 3pt
}{}{}

\begin{document}

\begin{center}
{\Large{\textbf{Twisted Yangians for symmetric pairs of types B, C, D}}} 

\bigskip

Nicolas Guay$^1$, Vidas Regelskis$^2$

\end{center}

\medskip

\begin{flushleft}
$^1$ University of Alberta, Department of Mathematics, CAB 632, Edmonton, \\ AB T6G 2G1, Canada. E-mail: nguay@ualberta.ca \\
$^2$ University of Surrey, Department of Mathematics, Guildford, GU2 7XH, \\ UK. E-mail: v.regelskis@surrey.ac.uk
\end{flushleft}

\smallskip

\begin{abstract}
We study a class of quantized enveloping algebras, called twisted Yangians, associated with the symmetric pairs of types B, C, D in Cartan's classification. These algebras can be regarded as coideal subalgebras of the Yangian for orthogonal or symplectic Lie algebras. They can also be presented as quotients of a reflection algebra by additional symmetry relations. We prove an analogue of the Poincar\'e--Birkhoff--Witt Theorem, determine their centres and also study extended reflection algebras. 
\end{abstract}

\smallskip

{\renewcommand{\baselinestretch}{0}
 \tableofcontents
}

\thispagestyle{empty}

\smallskip


\section{Introduction}

Twisted Yangians are some of the most elegant examples of the infinite dimensional reflection algebras introduced by E. Sklyanin in \cite{Sk}. The name twisted Yangian is due to G. Olshanskii, who constructed the first examples of such algebras for symmetric pairs of types AI and AII in \cite{Ol} using the $RTT$-presentation of Yangians \cite{FRT}. It is known that those twisted Yangians can be presented in two different ways: as abstract algebras defined by a reflection equation together with some additional relations, such as symmetry and unitarity relations, or as coideal subalgebras of the Yangian $Y(\mfgl_n)$. Twisted Yangians have also been shown to emerge in Drinfeld's original presentation of Yangians \cite{DMS}, an approach which allows the construction of generalized (or MacKay) twisted Yangians $Y(\mfg,\mfk)$ for symmetric pairs $(\mfg,\mfk)$ of arbitrary type \cite{Ma1}. Moreover, these twisted Yangians $Y(\mfg,\mfk)$ have been shown to be an integral part of many models of mathematical physics, such as open spin chains, vertex models, non-linear sigma models, and play an important part in quantum field theory; see e.g.\ \cite{Ma2} and references therein. 

The algebraic properties of Yangians of type A and twisted Yangians of types AI and AII (corresponding to the symmetric pairs $(\mfgl_N,\mfso_N)$ and $(\mfgl_N,\mfsp_N)$, or with $\mfgl_N$ replaced by $\mfsl_N$) were thoroughly explored in the survey paper \cite{MNO} by A.~Molev, M.~Nazarov and G.~Olshanskii (see also the references therein). The $RTT$-type relation gives the Yangian $Y(\mfgl_N)$, while $Y(\mfsl_N)$ is obtained by setting the quantum determinant of $Y(\mfgl_N)$ equal to the identity. For the case of the twisted Yangians of types AI and AII, the reflection equation (in its twisted form) leads to an extended twisted Yangian; by introducing an additional symmetry relation, the twisted Yangian is recovered. The analogue of the quantum determinant for the twisted Yangian is called the Sklyanin determinant. Its coefficients generate the whole centre of the twisted Yangian and, by setting it equal to a fixed scalar series (equal to $1$ in type AI but not in type AII; see Corollary 2.9.3 in \cite{Mo3}), the special twisted Yangian, which is a coideal subalgebra of $Y(\mfsl_N)$, is obtained. Finite-dimensional irreducible representations of these algebras were classified in \cite{Mo1} and their skew-representations were explored in \cite{Mo2}. Recent work of S. Khoroshkin and M. Nazarov (e.g. \cite{KhNa1,KhNa2, KhNa3, KNP} and related papers) provides explicit realizations of those representations for the Yangians of type A and for the twisted Yangians of types AI and AII using the theory of Howe dual pairs and Mickelsson algebras. There also exist symmetric pairs of type AIII, namely $(\mfgl_N, \mfgl_p \oplus \mfgl_{N-p})$: the corresponding twisted Yangians were constructed by A. Molev and E. Ragoucy in \cite{MoRa} who related them also to reflection algebras and classified their finite-dimensional irreducible representations. In this case, the reflection equation is used in its regular (non-twisted) form and the role of the symmetry equation is played by the unitarity constraint.

The $q$-analogues of twisted Yangians of types AI and AII were constructed in \cite{MRS} and of type AIII in \cite{CGM}. They can be called either twisted $q$-Yangians or twisted quantum loop algebras. Given a certain involution $\rho$ on $\mfg$, twisted Yangians can be understood as flat deformations of the enveloping algebra of the twisted (half-loop) current Lie algebra $\mfg[x]^\rho$ (see below for its definition), and their $q$-analogues are deformations of the enveloping algebra of the twisted loop Lie algebra $\mfg[x,x^{-1}]^\rho$ (where the involution on $\C[x,x^{-1}]$ is $x\mapsto x^{-1}$). Moreover, the defining relations of these algebras use a slightly different type of reflection equations. The specialization of quantum loop algebras to Yangians was postulated by Drinfel'd \cite{Dr1}, and was proven in \cite{GTL,GuMa}. The proof relies on Drinfeld's second presentation of these algebras \cite{Dr3}; however, no analogue of this presentation is known for twisted quantum loop algebras and twisted Yangians, but it is still possible to degenerate twisted quantum loop algebras to twisted Yangians using the $RTT$-presentation in types AI and AII \cite{CoGu}. Closure relations for twisted Yangians in Drinfeld's first presentation for symmetric pairs of general type were recently demonstrated in \cite{BeRe}. Twisted quantum loop algebras should fit in the more general framework of quantum symmetric pairs for Kac-Moody Lie algebras developed in \cite{Ko}. This last paper presents a generalization of the work of G.~Letzter \cite{Le}, M.~Noumi and T.~Sugitani on quantum symmetric spaces \cite{NoSu}. 

An $RTT$-presentation of Yangians associated with the classical Lie algebras of types B, C, D is given very explicitly by D.~Arnaudon {\it et al.} in \cite{AACFR}, but the existence of such presentations has been known since the foundational papers \cite{Dr1, Dr2}. It was further explored in \cite{AMR}, where certain isomorphisms between Yangians of low rank were constructed and the finite dimensional irreducible representations were classified. In this case, the $RTT$-type relation defines an extended Yangian $X(\mfg)$, while the Yangian $Y(\mfg)$ is obtained by taking the quotient of $X(\mfg)$ by the ideal generated by the coefficients of $z(u)-1$. Here the coefficients of the series $z(u)$ generate the centre of $X(\mfg)$.

The goal of this paper is to construct analogues of Olshanskii's twisted Yangians for all symmetric pairs of classical Lie algebras of types B, C, D and to describe fundamental properties of these new algebras. The symmetric pairs are those given by Cartan's classification of symmetric spaces (see \cite{He}, Chapter X):
$$
\text{BDI: } (\mfso_N,\mfso_p\op\mfso_q),\quad 
\text{CI: } (\mfsp_N,\mfgl_{N/2}), \quad
\text{CII: } (\mfsp_N,\mfsp_p\op\mfsp_q),\quad 
\text{DIII: } (\mfso_N,\mfgl_{N/2}),
$$
where $p+q=N$, and $p$, $q$, $N$ are all even in the CI, CII and DIII cases. For all of these cases, the twisted Yangian can be understood as a quantization of the universal enveloping algebra $\mfU\mfg[x]^\rho$ of the twisted current Lie algebra $\mfg[x]^\rho$ related to the pair $(\mfg,\mfg^\rho)$, where $\rho$ is an involution of $\mfg$ and $\mfg^\rho$ denotes the subalgebra of $\mfg$ fixed by $\rho$. The twisted current algebra $\mfg[x]^\rho$ is defined as the subspace of $\mfg[x]$ consisting of elements fixed by the involution $\rho$ extended to $\mfg[x]$ by $\rho(F p(x)) = \rho(F) p(-x) \; \forall \, F\in\mfg,\, \forall \, p(x) \in \C[x]$.

We also construct twisted Yangians corresponding to trivial symmetric pairs, namely
$$
\text{BCD0: } (\mfg,\mfg) \quad\text{for}\quad \mfg=\mfsp_N \quad\text{and}\quad \mfg=\mfso_N.
$$
In this case, the involution $\rho$ acts trivially on $\mfg$, but is non-trivially extended to the current Lie algebra, giving $\mfg[x]^\rho=\mfg[x^2]$. Despite the fact that $\mfg[x^2]\cong\mfg[x]$ as a Lie algebra, the quantization of the twisted current algebra $\mfU\mfg[x^2]$ is a twisted Yangian not isomorphic to $Y(\mfg)$. For Lie algebras of type A, the corresponding twisted Yangian can be constructed as in \cite{MoRa} for the extremal case $p=N$ of the symmetric pair of type AIII $(\mfgl_N,\mfgl_{p}\op\mfgl_{N-p})$. For symmetric pairs of types BDI and CII, we can also set $p=N$ and $q=0$. However, as we will see in this paper, there are some important differences between the extremal and non-extremal cases. This is in contrast to type AIII, where all of the twisted Yangians with $p=0,\ldots, N$ obey relations of the same form. Our twisted Yangians of type BCD0 are very similar to the reflection algebra defined in \cite{IMO} - see Definition 3.1, Proposition 3.2 and the homomorphism (3.44) in \textit{loc.~cit.}

First, we define extended twisted Yangians $X(\mfg,\mcG)^{tw}$ as coideal subalgebras of the extended Yangians $X(\mfg)$, and twisted Yangians $Y(\mfg,\mcG)^{tw}$ as quotients of the extended twisted Yangians by the unitarity constraint. The latter have no non-trivial central elements and are coideal subalgebras of the Yangians $Y(\mfg)$. The construction of these coideal subalgebras is based on a matrix $\mcG(u)$, which is a solution of the reflection equation and is a rational function of the spectral parameter $u$ and the matrix $\mcG$. The corresponding symmetric pair is $(\mfg,\mfg^\rho)$ where $\mfg^\rho=\{X\in\mfg\,|\,X=\mcG X \mcG^{-1}\}$, and all the symmetric pairs of types B, C and D can be obtained this way. Moreover, the form of the rational matrix $\mcG(u)$ coincides with that of rational $K$-matrices of the principal chiral model on the half-line found in \cite{MaSho}. The differences are due to the fact that for a given symmetric pair, the matrix $\mcG$ is not unique. 

We show that $Y(\mfg,\mcG)^{tw}$ is isomorphic to a subalgebra $\wt{Y}(\mfg,\mcG)^{tw}$ of the extended Yangian and this leads to the decomposition $X(\mfg,\mcG)^{tw} \cong ZX(\mfg,\mcG)^{tw} \otimes \wt{Y}(\mfg,\mcG)^{tw}$ where $ZX(\mfg,\mcG)^{tw}$ is the centre of $X(\mfg,\mcG)^{tw}$: see Theorem \ref{T:Yt=Xt^Y}. We then prove an analogue of the Poincar\'e--Birkhoff--Witt Theorem for the twisted Yangians and their extended version (Theorem \ref{Y:PBW}), and determine the centre of $X(\mfg,\mcG)^{tw}$ in Subsection \ref{Sec:34}. We also explain how twisted Yangians provide a quantization of a left Lie coideal structure: see Theorem \ref{T:quan}. 

In Section \ref{S:ra}, we show that twisted Yangians (extended or not) are isomorphic to a class of reflection algebras satisfying additional symmetry and unitarity relations: see Theorems \ref{T:Y(g,G)} and \ref{T:BGdecomp}. In the following section on the quantum contraction, we introduce extended reflection algebras and explain how the symmetry and unitarity relations are equivalent to the vanishing of certain central elements: see Theorems \ref{T:XRES} and \ref{T:d(u)}. These central elements are obtained as coefficients of certain even and odd power series. Similar results were already known for twisted Yangians of types AI and AII \cite{MNO}.

Since $\mfsl_2$ is isomorphic to $\mfso_3$ and $\mfsp_2$, it is natural to ask if our twisted Yangians for $\mfg = \mfso_3$ or $\mfg = \mfsp_2$ are isomorphic to Olshanski's twisted Yangians for $\mfsl_2$: this is indeed the case as proved in \cite{GRW}.

In a future work, we hope to explore $q$-analogues of our twisted Yangians and of the extended Yangians. 

A word of explanation is necessary to clarify the terminology used in this paper. We use the name {\it twisted Yangian} when referring to coideal subalgebras of a {\it Yangian}. We use the name {\it reflection algebra} for algebras defined by a {\it reflection equation}. {\it Twisted Yangians} and {\it reflection algebras} are not isomorphic in general; they become isomorphic by requiring additional (symmetry and/or unitarity) relations to hold.

\smallskip

{\bf Acknowledgements.} The first author acknowledges the support of the Natural Sciences and Engineering Research Council of Canada through its Discovery Grant program. Part of this work was done during the second author's visits to the University of Alberta. V.R. thanks the University of Alberta for the hospitality, and also gratefully acknowledges the Engineering and Physical Sciences Research Council of the United Kingdom for the Postdoctoral Fellowship under the grant EP/K031805/1.


\section{Extended Yangian}

The extended Yangian $X(\mfg)$ was first introduced in \cite{AACFR} and was studied furthermore in \cite{AMR}. It admits as quotients the standard (untwisted) orthogonal and symplectic Yangians $Y(\mfg)$: see the remark at the end of Section 2 in \cite{AMR}. In this section, we will summarize relevant definitions and results from \textit{loc.~cit.}, to which we refer the reader for detailed explanations. 

Let $n \in\N$ and set $N=2n$ or $N=2n+1$. Then $\mfg$ will denote either the orthogonal Lie algebra $\mfso_N$ or the symplectic Lie algebra $\mfsp_N$ (only when $N=2n$); in the orthogonal case we assume $N\ge3$, as in \cite{AMR}, so that $\mfg$ is semi-simple and $\ka\neq0$. These algebras can be realized as Lie subalgebras of $\mfgl_N$ in the following way. Let us label the rows and columns of $\mfgl_N$ by the indices $\{ -n,\ldots,-1,1,\ldots,n\}$ if $N=2n$ and by $\{-n,\ldots,-1,0,1,\ldots,n \}$ if $N=2n+1$. Set $\theta_{ij}=1$ in the orthogonal case $\forall \, i,j$ and $\theta_{ij}=\mathrm{sign}(i)\cdot \mathrm{sign}(j)$ in the symplectic case for $i,j\in\{ \pm 1, \pm 2, \ldots, \pm n \}$. Let $F_{ij}=E_{ij} - \theta_{ij} E_{-j,-i}$ where $E_{ij}$ is the usual elementary matrix of $\mfgl_N$. Then $\mfg = \mathrm{span}_{\C} \{ F_{ij} \, | \, -n\le i,j\le n \}$. These matrices satisfy the relations 
\eq{ \label{[F,F]}
F_{ij} + \theta_{ij}F_{-j,-i}=0, \qquad [F_{ij},F_{kl}] = \delta_{jk}F_{il} - \delta_{il}F_{kj} + \theta_{ij}\delta_{j,-l}F_{k,-i} - \theta_{ij}\delta_{i,-k}F_{-j,l}.
}

All the tensor products in this paper will be over $\C$, so $\ot = \ot_{\C}$. We need to introduce some operators: $P \in \End\,\C^N \ot\End\,\C^N$ will denote the permutation operator on $\C^N \ot \C^N$, and $Q$ will denote the transposed projector on $\C^N \ot \C^N$, so
\eq{ \label{PQ}
P=\mysum_{i,j=-n}^n E_{ij} \ot E_{ji}, \qquad Q = \mysum_{i,j=-n}^n \theta_{ij} E_{ij} \ot E_{-i,-j} .
}
The operator $Q$ is obtained from $P$ by taking the transpose of either the first or the second matrix, namely $Q=P^{t_1}=P^{t_2}$, the transpose $t$ being the one with respect to the bilinear form on $\C^N$ given by $(u,v) = u' \mcB v$ where $\mcB$ is the matrix with entries $b_{ij}=\mathrm{sign}(i)\,\delta_{i,-j}$ in the symplectic case and $b_{ij}=\delta_{i,-j}$ in the orthogonal case, and the primed notation $u'$ denotes the usual matrix transposition. The transposition $t$ acts on the basis elements by the rule $(E_{ij})^t = \theta_{ij} E_{-j,-i}$. Let $I$ denote the identity matrix. Then $P^2=I$, and also $PQ=QP=\pm Q$ and $Q^2=N Q$, which will be useful below. Here (and further in this paper) the upper sign corresponds to the orthogonal case and the lower sign to the symplectic case. 

Set $\ka=N/2\mp 1$. The $R$-matrix $R(u)$ that we will need is defined by:
\eq{
R(u) = I - \frac{P}{u} + \frac{Q}{u-\ka}. \label{R(u)}
}
It is a solution of the quantum Yang-Baxter equation with spectral parameter,
\eq{ \label{YBE}
R_{12}(u)\, R_{13}(u+v)\, R_{23}(v) = R_{23}(v)\, R_{13}(u+v)\, R_{12}(u).
} 
We borrowed the matrix $R(u)$ from \cite{AACFR}, but it actually appeared earlier in \cite{ZaZa} and \cite{KuSk}.

\begin{defn} [\cite{AACFR,AMR}]
The extended Yangian $X(\mfg)$ is the associative $\C$-algebra with generators $t_{ij}^{(r)}$ for $-n\le i,j\le n$ and $r\in\Z_{\ge 0}$, which satisfy the following relations: 
\eq{ \label{RTT}
R(u-v)\,T_1(u)\,T_2(v) = T_2(v)\,T_1(u)\,R(u-v), 
}
where $T_1(u)$ and $T_2(u)$ are the elements of $ \End\,\C^N \ot \End\,\C^N \ot X(\mfg)[[u^{-1}]] $ given by
\eq{
T_1(u) = \mysum_{i,j=-n}^n E_{ij} \ot 1 \ot t_{ij}(u) , \qquad
T_2(u) = \mysum_{i,j=-n}^n 1 \ot E_{ij} \ot t_{ij}(u) , \nn
}
with the formal power series given by 
$$
t_{ij}(u) = \mysum_{r=0}^{\infty} t_{ij}^{(r)}\, u^{-r} \in X(\mfg)[[u^{-1}]], \qquad
t_{ij}^{(0)} = \delta_{ij}.
$$
In terms of the power series elements $t_{ij}(u)$, the defining relations are
\eqa{ \label{[t,t]}
[\, t_{ij}(u),t_{kl}(v)]&=\frac{1}{u-v}
\Big(t_{kj}(u)\, t_{il}(v)-t_{kj}(v)\, t_{il}(u)\Big)\el
{}&-\frac{1}{u-v-\ka}
\mysuml_{a=-n}^n\Big(\delta_{k,-i}\,\theta_{ia}\, t_{aj}(u)\, t_{-a,l}(v)-
\delta_{l,-j}\,\theta_{ja}\, t_{k,-a}(v)\, t_{ia}(u)\Big). 
}
The Hopf algebra structure of $X(\mfg)$ is given by 
\eq{ \label{Hopf X(g)}
\Delta: t_{ij}(u)\mapsto \mysum_{k=-n}^n t_{ik}(u)\ot t_{kj}(u), \qquad S: T(u)\mapsto T^{-1}(u),\qquad \epsilon: T(u)\mapsto I. 
}
\end{defn}

\begin{lemma} [{\cite[Proposition 3.11]{AMR}}]
There exists an embedding $\mfU\mfg \into X(\mfg)$ of the enveloping algebra of $\mfg$ into the extended Yangian given by $F_{ij} \mapsto \frac{1}{2} (t_{ij}^{(1)} - \theta_{ij} t_{-j,-i}^{(1)})$.
\end{lemma}

\begin{rmk} [\cite{AMR}] 
If $i\neq j$, then $t_{ij}^{(1)} = -\theta_{ij} t_{-j,-i}^{(1)}$, so the embedding sends $F_{ij}$ to $t_{ij}^{(1)}$. However, $t^{(1)}_{ii} = z_1 - t_{-i,-i}^{(1)}$ where $z_1$ is a certain central element in $X(\mfg)$, so the previous embedding maps \mbox{$F_{ii}$ to $t_{ii}^{(1)} \!-\! \frac{z_1}{2}$.} 
\end{rmk}

Next, we will state some properties of $X(\mfg)$ that we will require in further sections. Consider an arbitrary formal series $f(u)$ of the form
\eq{
f(u)=1+f_1u^{-1}+f_2 u^{-2}+\cdots\in \C[[u^{-1}]]. \nn
}
Let $a\in\C$ be an arbitrary constant and let $A$ be a matrix with entries in $\C$ such that $AA^t=I$. Then each of the maps in the first line below defines an automorphism of $X(\mfg)$ and each map in the second line defines an anti-automorphism:
\eqa{
\mu_f:T(u)&\mapsto f(u)\, T(u), &\tau_a:T(u)&\mapsto T(u-a), &\al_A:T(u)&\mapsto A\, T(u) A^{t}, \label{Aut}\\
T(u)&\mapsto T(-u), & T(u)&\mapsto T^{t}(u), &T(u)&\mapsto T^{-1}(u). \label{Anti-Aut}
}
This is verified with the use of the following property of the $R$-matrix:
\eq{ \label{R(u)R(-u)}
R(u)\, R(-u)=(1-u^{-2})\cdot I,
}
and the fact that $R(u)$ is stable under the composition of the transpositions in the first and the second copies of $\End\,\C^N$: $R^{t_1t_2}(u)=R(u)$.

Let $ZX(\mfg)$ denote the centre of $X(\mfg)$. Multiply both sides of \eqref{RTT} by $u-v-\ka$ and set $u=v+\ka$. Then upon replacing $v$ by $u$ one obtains $Q\,T_1(u+\ka)\,T_2(u) = T_2(u)\,T_1(u+\ka)\,Q$. Recall that $N^{-1}Q$ is a projection operator to a one-dimensional subspace of $\C^N\ot\C^N$. Thus the expression above must be equal to $Q$ times a formal power series $z(u)$. Using the definition of $Q$, one deduces that $Q\,T_1(u)=Q\,T_2^t(u)$ and $T_1(u)\,Q=T_2^t(u)\,Q$. From this, one can show that
\eq{
T^{t}(u+\ka)\,T(u) = T(u)\,T^t(u+\ka) = z(u)\cdot I, \label{TT}
}
where $z(u)=1+\mysum_{i\geq1} z_i\,u^{-i}$ is called the quantum contraction of the matrix $T(u)$; its coefficients $z_i$ generate the centre $ZX(\mfg)$ of $X(\mfg)$. This leads to the following tensor product decomposition of $X(\mfg)$ \cite[Theorem~3.1]{AMR}:
\eqa{
X(\mfg) = ZX(\mfg)\ot Y(\mfg) , \label{X=ZX*Y}
}
where $Y(\mfg)$ is the Yangian of $\mfg$. $Y(\mfg)$ is thus isomorphic to the quotient of $X(\mfg)$ by the ideal generated by the central elements $z_i$, that is, $Y(\mfg)\cong X(\mfg)/(z(u)-1)$. It is also isomorphic to the subalgebra of $X(\mfg)$ stable under all the automorphisms $\mu_f$. Let us give a few more details which will be relevant later.

Let $y(u)$ be the unique series such that $z(u)=y(u)\,y(u+\ka)$. By \eqref{TT} the automorphism $\mu_f$ takes $y(u)$ to $f(u)\,y(u)$. The Yangian $Y(\mfg)$ (\cite{AMR}, Corollary 3.2) may be alternatively defined as the subalgebra of $X(\mfg)$ stable under all the automorphisms $\mu_f$ given in \eqref{Aut}, i.e.\ as the subalgebra $\wt{Y}(\mfg)$ generated by the coefficients $\tau^{(r)}_{ij}$ of the series $\tau_{ij}(u) = y^{-1}(u)\,t_{ij}(u)$ with $-n\le i,j\le n$ and $r\in\Z_{\ge0}$.

The generators $\tau_{ij}^{(r)}$ of $\wt{Y}(\mfg)$ satisfy the relations \eqref{[t,t]} with $t_{ij}(u)$ replaced by $\tau_{ij}(u)$ and the additional relation
\eq{ \label{tt}
\mysum_{a=-n}^n \theta_{ak}\, \tau_{-a,-k}(u+\kappa)\,\tau_{al}(u) = \delta_{kl}.
}
We can also express these as:
\eq{
T(u) = y(u)\,\mcT(u) , \qquad \mcT(u)\,\mcT^t(u+\ka) = \mcT^t(u+\ka)\,\mcT(u) = I \label{cTT}
} 
where $\mcT(u)$ is the matrix with entries $\tau_{ij}(u)$.

\section{Twisted Yangians}

The twisted Yangians of types AI and AII, corresponding to the symmetric pairs $(\mfgl_N,\mfso_N)$, and $(\mfgl_N,\mfsp_N)$ and the twisted reflection equation were first introduced by G. Olshanskii in \cite{Ol} and have been studied extensively over the past twenty years (see e.g.\ \cite{MNO} for a pedagogic exposition). Those of type AIII were first investigated in \cite{MoRa} where they were called reflection algebras since they can be defined using the non-twisted reflection equation, and their twisted quantum loop analogues were introduced in \cite{CGM}. In this section, we introduce new twisted Yangians for the classical Lie algebras of types B, C and D: they are in bijection with the symmetric pairs of types BDI, CI, CII and DIII. This notation refers to Cartan's classification of symmetric spaces. We also introduce twisted Yangians BCD0 of even levels that are analogues of the even loop twisted Yangians of \cite{BeRe} and the reflection algebras $\mcB(n,0)$ of \cite{MoRa}.

\subsection{Symmetric pairs of types B, C, D}

The symmetric pairs we are interested in are of the form $(\mfg,\mfg^{\rho})$ where $\rho$ is an involutive automorphism of $\mfg$ given by $\mathrm{Ad}(\mcG)$ and
\[
G = \{ A\in SL_N(\C) \, | \, A^{-1} = A^t \} \text{ and } \mfg = \{ X\in\mfsl_N \, | \, X + X^t=0 \}.
\]
For each pair we choose a matrix $\mcG$ with $\mcG^2=I$. It satisfies $\mcG^t=(\pm)\mcG$, with $+$ in types BCD0, BDI and CII and $-$ in types CI and DIII. These properties imply that $\rho=\mathrm{Ad}(\mcG)$ is an involutive automorphism of $\mfg$. The fixed-point subalgebra $\mfg^{\rho}$ is given by $\mfg^{\rho} = \{ X\in\mfg \, | \, X = \mcG X \mcG^{-1} \} = \mathrm{span}\{ X + \mcG X \mcG^{-1} \, | \, X\in\mfg \}$. We will denote by $\check{\mfg}^{\rho}$ the eigenspace of eigenvalue $-1$ of $\rho$ and by $g_{ij}$ the entries of $\mcG$. The matrix $\mcG$ is not unique: the matrices listed below are only representatives (see the end of this subsection).

\smallskip

Let us consider each symmetric pair and one or two choices for the matrix $\mcG$:
\vspace{-1ex}
\begin{itemize} [itemsep=0ex]

\item BCD0\,: $\mcG=I_N$, $\rho$ is trivial and $\mfg^{\rho} = \mfg$.

\item CI{\phantom{IIII}} : $N$ is even, $\mfg=\mfsp_N$, $\mcG=\mysum_{i=1}^{\frac{N}{2}} (E_{ii} - E_{-i,-i})$ and $\mfg^{\rho} \cong \mfgl_{\frac{N}{2}}$. In this case $\sqrt{-1}\mcG\in G$.

\item DIII{\phantom{II}} : $N$ is even, $\mfg=\mfso_N$, $\mcG=\mysum_{i=1}^{\frac{N}{2}} (E_{ii} - E_{-i,-i})$ and $\mfg^{\rho} \cong \mfgl_{\frac{N}{2}}$. In this case $\sqrt{-1}\mcG\in G$.

\item CII{\phantom{IIII}}: $N$, $p$ and $q$ are even and $>0$, $N=p+q$, $p\ge q$, $\mfg=\mfsp_N$, 
\[ 
\mcG= -\mysum_{i=1}^{\frac{q}{2}} (E_{ii} + E_{-i,-i}) + \mysum_{i=\frac{q}{2}+1}^{\frac{N}{2}} (E_{ii} + E_{-i,-i}) 
\] 
and $\mfg^{\rho}= \mfsp_p \op \mfsp_q$. More precisely, the subalgebra of $\mfg^{\rho}$ spanned by $F_{ij}$ with $-\frac{q}{2} \le i,j \le \frac{q}{2}$ is isomorphic to $\mfsp_q$ and the subalgebra of $\mfg^{\rho}$ spanned by $F_{ij}$ with $|i|,|j| > \frac{q}{2}$ is isomorphic to $\mfsp_p$. 

\item BDI{\phantom{II}} : $\mfg=\mfso_N$, $\mfg^{\rho}= \mfso_p \op \mfso_q$ where $p>q>0$ if $N$ is odd, and $p\geq q>0$ if $N$ is even. (If $q=1$, then $\mfso_q$ is the zero Lie algebra.) We write BI when $N$ is odd and DI when $N$ is even.
When $N$ is even, $p$ and $q$ have the same parity and $\mcG$ is given by 
\[ 
\mcG=\mysum_{i=1}^{\frac{p-q}{2}} (E_{ii} + E_{-i,-i}) + \mysum_{i=\frac{p-q}{2}+1}^{\frac{N}{2} } (E_{-i,i} + E_{i,-i}). 
\] 
When $N$ is odd, $p-q$ is odd and 
\[ 
\mcG=\mysum_{i=-\frac{p-q-1}{2}}^{\frac{p-q-1}{2}} E_{ii}
 + \mysum_{i=\frac{p-q+1}{2}}^{\frac{N-1}{2}} (E_{-i,i} + E_{i,-i}). 
 \]
We have $\det\mcG=(-1)^q$. Thus $\mcG\in G$ if $q$ is even, and $-\mcG\in G$ if $q$ and $N$ are odd. If $q$ is odd and $N$ is even, no scalar multiple of $\mcG$ lies in $G$, and $\rho$ is outer.

\end{itemize}

To see that $\mfg^{\rho} \cong \mfso_p \op \mfso_q$, we will adopt the more common point of view on $\mfso_N$, namely that it is isomorphic to the Lie algebra of matrices in $\mfsl_N$ skew-symmetric with respect to the main diagonal. 

Let $\wt{\mfso}_N = \{ X\in\mfsl_N \, | \, X= - X' \}$: here, $X'$ is the standard transpose of $X$ with respect to its main diagonal. Let $C$ be the matrix with non zero-entries given by $c_{ii} = -\frac{\sqrt{-1}}{\sqrt{2}}, c_{-i,-i} =\frac{1}{\sqrt{2}}, c_{-i,i}=\frac{\sqrt{-1}}{\sqrt{2}}, c_{i,-i} =\frac{1}{\sqrt{2}}$ for $1\le i \le \frac{N}{2}$ if $N$ is even and for $1\le i \le \frac{N-1}{2}$ if $N$ is odd; in the latter case, we also set $c_{00}=1$. Then $CC' = K$ where $K$ is the antidiagonal matrix with entries $k_{ij} = \delta_{i,-j}$. An isomorphism $\varphi: \wt{\mfso}_N \lra \mfg$ is given by $\varphi(X) = C X C^{-1}$. Indeed, if $X=-X'$, then $-\varphi(X)^t = -K(C X C^{-1})' K = - K (C^{-1})' X' C' K = C X C^{-1} = \varphi(X)$, so $\varphi(X)\in\mfg$. 

If $N$ is even and $p> q$, we let $\wt{\mcG}$ be the diagonal matrix with entries $\wt{g}_{ii}=1$ for $-\frac{N}{2}\le i\le p-\frac{N}{2}$ and $\wt{g}_{ii}=-1$ for $p-\frac{N}{2}+1\le i\le \frac{N}{2}$. If $N$ is even and $p=q=\frac{N}{2}$, we let $\wt{\mcG}$ be the diagonal matrix with entries $\wt{g}_{ii}=1$ for $i<0$ and $\wt{g}_{ii}=-1$ for $i>0$. If $N$ is odd and $p> q$, we let $\wt{\mcG}$ be the diagonal matrix with entries $\wt{g}_{ii}=1$ for $-\frac{N-1}{2}\le i\le p-\frac{N+1}{2}$ and $\wt{g}_{ii}=-1$ for $p-\frac{N-1}{2}\le i\le \frac{N-1}{2}$. Conjugation by $\wt{\mcG}$ defines an automorphism $\wt{\rho}$ of $\wt{\mfso}_N$ with fixed-point subalgebra isomorphic to $\mfso_p \op \mfso_q$. Using $\varphi$, we can transport it to an automorphism $\rho$ of $\mfg$: $\rho(X) = (\varphi \circ \wt{\rho} \circ\varphi^{-1})(X) = (C\wt{\mcG}C^{-1})X(C\wt{\mcG}C^{-1})^{-1}$. Observe that $\mcG = C\wt{\mcG}C^{-1}$ with $\mcG$ as given above, and $\rho(X) = \mcG X\mcG^{-1}$: this proves that $\mfg^{\rho} \cong \mfso_p \op \mfso_q$ because $\mfg^{\rho}$ is isomorphic via $\varphi$ to $\wt{\mfso}_N^{\wt{\rho}}$. 

When $N$, $p$ and $q$ are even, another possibility for $\mcG$ is $\mysum_{i=\frac{q}{2}+1}^{\frac{N}{2}} (E_{ii} + E_{-i,-i}) - \mysum_{i=1}^{\frac{q}{2}} (E_{ii} + E_{-i,-i})$. When $N$ is odd, $p$ is odd and $q$ is even, another possibility for $\mcG$ is $E_{00} + \mysum_{i=1}^{\frac{p-1}{2}} (E_{ii} + E_{-i,-i}) - \mysum_{i=\frac{p+1}{2}}^{\frac{N-1}{2}} (E_{ii} + E_{-i,-i})$. The fixed-point subalgebra is also isomorphic to $\mfso_p \op \mfso_q$. The main advantage of the first matrix $\mcG$ given above in the BDI case is that it works for all possible parities of $N$, $p$ and $q$.

The various matrices $\mcG$ chosen in the previous paragraphs will help us define the twisted Yangians that will be of interest to us in the remainder of this article. They are not the only ones that we could use. We only need covariance. Let $D\in G$ and $\mcG'=D\,\mcG\,D^{-1}$. Then conjugation by $D$ maps $\rho$, $\mfg^{\rho}$ and $\check{\mfg}^{\rho}$ to $\rho'=\mathrm{Ad}(\mcG')$, $\mfg^{\rho'}$ and $\check{\mfg}^{\rho'}$. Since $\mathrm{tr}\,\mcG'=\mathrm{tr}\,\mcG$, it also maps the matrix $\mcG(u)$ of Definition \ref{D:X(g,G)}, and its reflection equation, to those for $\mcG'$. Theorem 6.1 in \cite{He} says that if $\rho_1$ and $\rho_2$ are two involutions of a simple Lie algebra $\mfg$ and if $\mfg^{\rho_1}$ is isomorphic to $\mfg^{\rho_2}$, then $\rho_1$ and $\rho_2$ are conjugate under $\mathrm{Aut}(\mfg)$. When $\mfg$ is of type B or C, there are no Dynkin diagram automorphisms and consequently $\mathrm{Aut}(\mfg)$ consists of inner automorphisms. Write $\rho_i=\mathrm{Ad}(\mcG_i)$ with $\mcG_1,\mcG_2\in G$. This means that there exists a third matrix $D$ in $G$ such that $\mathrm{Ad}(\mcG_1) = \mathrm{Ad}(D) \mathrm{Ad}(\mcG_2) \mathrm{Ad}(D)^{-1}$, hence $\mcG_1 = Z D \mcG_2 D^{-1}$ where $Z$ is in the centre of $G$. We can take $\mcG_2$ to be $(-1)^q\mcG$ (in type BI), $\mcG$ (in type CII) or $\sqrt{-1}\mcG$ (in type CI), with $\mcG$ as above; these lie in $G$, and conclude that if $\mcG_1$ is any other matrix in $G$ such that $\mathrm{Ad}(\mcG_1)$ is an involution of $\mfg$ with fixed-point subalgebra isomorphic to $\mfg^{\rho}$, then $\mcG_1$ is in the orbit of $\mcG_2$ under the adjoint action of $G$, up to multiplication by a central element in $G$. (The centre is trivial when $G=SO_N(\C)$ and $N$ is odd, and it is equal to $\{ \pm I \}$ when $G=SO_N(\C)$ with $N$ even or $G = Sp_N(\C)$.) The orbit of $\mcG$ (or $\sqrt{-1}\mcG$) under the action of $G$ is in bijection with $G/\mathrm{Cent}_G(\mcG)$ where $\mathrm{Cent}_G(\mcG)$ is the centralizer of $\mcG$ in $G$. The centralizer $\mathrm{Cent}_G(\mcG)$ can be determined for the specific matrices $\mcG$ considered above and is a complex Lie subgroup of $G$ with Lie algebra $\mfg^{\rho}$. In type $D$, however, isomorphic fixed-point subalgebras need not give conjugate matrices, even up to a scalar factor. For $N=8$, the involutions of types DIII and DI with $(p,q)=(6,2)$ both have fixed-point subalgebra isomorphic to $\mfgl_4$, since $\mfso_6\op\mfso_2\cong\mfsl_4\op\C\cong\mfgl_4$. The eigenvalue $1$ has multiplicities $4$ and $6$ in the corresponding matrices $\mcG$, so these matrices are not conjugate in $GL_8(\C)$, even up to a scalar factor. This exception arises from triality in type $D_4$.


\subsection{Twisted Yangians as subalgebras and quotients of extended twisted Yangians}

We now introduce two types of twisted Yangians associated to the extended Yangian $X(\mfg)$. We will explore their algebraic structure in the subsections below. 

\begin{defn} \label{D:X(g,G)}
Let the matrix $\mcG$ be as described above. The extended twisted Yangian $X(\mfg,\mcG)^{tw}$ is the subalgebra of $X(\mfg)$ generated by the coefficients of the entries of the $S$-matrix 
\eq{
S(u) = T(u-\ka/2)\,\mcG(u)\,T^t(-u+\ka/2) , \label{S=TGT} \vspace{-1ex}
}
where
\vspace{-.75ex}
\begin{itemize} [itemsep=-0.1ex]
\item $\mcG(u)=\mcG$ for cases BCD0, CI, DIII and DI, CII when $p=q$;
\item $\mcG(u) = (I-c\,u\, \mcG)(1-c\,u)^{-1}$ with $c=\tfrac{4}{p-q}$ for cases BDI, CII when $p>q$.
\end{itemize}
We will further refer to the first case above as `$\mcG$ of the first kind' and to the second case as `$\mcG$ of the second kind'. We will use the same terminology for the matrices $\mcG(u)$ and $S(u)$. 
\end{defn}

\begin{rmk}
The BCD0 case was considered in \cite{IMO} (see Definition 3.1 of their reflection algebra); those authors used a slightly different formula for $S(u)$ (see Proposition 3.2 in \cite{IMO}). Our formula \eqref{S=TGT} is more in line with the one used in \cite{MNO}. The choice of $\mcG(u)$ is motivated by \cite{MaSho}. The rational form of $\mcG(u)$ is a new feature of twisted Yangians of types B,C,D which is not present in type A. The shift by $\ka/2$ in \eqref{S=TGT} is imposed upon us by \eqref{TT}. Shifting by $\ka/2$ gives a more symmetric formula for $S(u)$ and for the left-hand side of \eqref{Suwu}, which is similar to the notation used in physics for the unitary condition.
\end{rmk}

\begin{prop}\label{sususcalar}
In the algebra $X(\mfg,\mcG)^{tw}$, the product $S(u)\,S(-u)$ is a scalar matrix
\eq{
S(u)\,S(-u) = w(u)\cdot I, \label{Suwu}
}
where $w(u)$ is an even formal power series in $u^{-1}$ with coefficients $w_i$ $(i=2,4,\ldots)$ central in $X(\mfg,\mcG)^{tw}$.
\end{prop}

\begin{proof}
Recall that $T^{t}(u+\ka)\,T(u)=T(u)\,T^{t}(u+\ka)=z(u)\cdot I$. Thus
$$
S(u)\,S(-u) = z(-u-\ka/2)\,z(u-\ka/2)\cdot I ,
$$
and $w(u)=z(-u-\ka/2)\,z(u-\ka/2)$ is indeed an even series whose coefficients are central since so are the coefficients of $z(u)$ in $X(\mfg)$.
\end{proof}

Let $W(\mfg,\mcG)^{tw}$ denote the commutative algebra generated by the coefficients of $w(u)$. It will be proven in Section \ref{Sec:34} that $W(\mfg,\mcG)^{tw}$ is indeed the centre of $X(\mfg,\mcG)^{tw}$.

\begin{defn} \label{D:Y(g,G)}
The twisted Yangian $Y(\mfg,\mcG)^{tw}$ is the quotient of $X(\mfg,\mcG)^{tw}$ by the ideal generated by the coefficients of the unitarity relation, i.e.,
\eq{
Y(\mfg,\mcG)^{tw} = X(\mfg,\mcG)^{tw} / (S(u)\,S(-u) - I) . \label{Y=X/(SS-I)}
}
\end{defn}

The new Yangians, as Olshanskii's twisted Yangians, are coideal subalgebras of a larger Yangian.

\begin{prop} \label{P:coideal} 
The algebra $X(\mfg,\mcG)^{tw}$ is a left coideal subalgebra of $X(\mfg)$:
\eq{
\Delta(X(\mfg,\mcG)^{tw}) \subset X(\mfg) \ot X(\mfg,\mcG)^{tw}. \nn
}
\end{prop}

\begin{proof}
It is sufficient to show that $\Delta (s_{ij}(u)) \in X(\mfg) \ot X(\mfg,\mcG)^{tw}$. Indeed, 
$$
s_{ij}(u) = \mysum_{a,b=-n}^n \theta_{jb}\, t_{ia}(u-\ka/2)\,g_{ab}(u)\,t_{-j,-b}(-u+\ka/2). \vspace{-.75ex}
$$
and by \eqref{Hopf X(g)}, \vspace{-.75ex}
\eqa{ \label{cop:s(u)}
\Delta(s_{ij}(u)) = \mysum_{a,b=-n}^n \theta_{jb}\,t_{ia}(u-\ka/2)\,t_{-j,-b}(-u+\ka/2)\ot s_{ab}(u),
}
which completes the proof.
\end{proof}

The series $w(u)$ is group-like in $X(\mfg,\mcG)^{tw}$, that is, $\Delta : w(u) \mapsto w(u)\ot w(u)$. This follows straightforwardly since $\Delta : z(u) \mapsto z(u)\ot z(u)$ (see (2.29) in \cite{AMR}), which can be obtained from \eqref{TT}.

\smallskip

Now we show that $Y(\mfg,\mcG)^{tw}$ is isomorphic to a subalgebra of the extended twisted Yangian. Recall that $\wt Y(\mfg)$ was defined as the subalgebra of $X(\mfg)$ generated by the coefficients $\tau^{(r)}_{ij}$ of the series $\tau_{ij}(u) = y^{-1}(u)\,t_{ij}(u)$ with $-n\le i,j\le n$ and $r\in\Z_{\ge0}$; moreover, the $\tau_{ij}(u)$ are matrix entries of $\mcT(u)$.

\begin{thrm} \label{T:Yt=Xt^Y}
Let $\wt{Y}(\mfg,\mcG)^{tw}$ be the subalgebra of $\wt Y(\mfg)$ generated by the coefficients $\si_{ij}(u)$ of $\Si(u)$ defined by $\Si(u) = \mcT(u-\ka/2) \mcG(u) \mcT^t(-u+\ka/2)$. Then $\wt{Y}(\mfg,\mcG)^{tw}$ is a subalgebra of $X(\mfg,\mcG)^{tw}$ and the quotient homomorphism $X(\mfg,\mcG)^{tw} \onto Y(\mfg,\mcG)^{tw}$ induces an isomorphism between $\wt{Y}(\mfg,\mcG)^{tw}$ and $Y(\mfg,\mcG)^{tw}$. Moreover, $X(\mfg,\mcG)^{tw}$ is isomorphic to $W(\mfg,\mcG)^{tw} \otimes \wt{Y}(\mfg,\mcG)^{tw}$. 
\end{thrm}

\begin{proof}
Set $q(u) = y(u-\ka/2)y(-u+\ka/2)$. Then $\Sigma(u) = q(u)^{-1} S(u) $ and $w(u) = q(u)q(-u) = q(u)q(u+\ka)$. It follows from the last equality using induction that the coefficients of $q(u)$ can be expressed in terms of the coefficients of $w(u)$, hence belong to the centre of $X(\mfg,\mcG)^{tw}$. The entries of $\Si(u)$ are thus also in $X(\mfg,\mcG)^{tw}$, so $\wt{Y}(\mfg,\mcG)^{tw}$ is a subalgebra of $X(\mfg,\mcG)^{tw}$. From the decomposition $S(u) = q(u) \Sigma(u)$, it follows that $X(\mfg,\mcG)^{tw} \cong W(\mfg,\mcG)^{tw} \cdot \wt{Y}(\mfg,\mcG)^{tw} $. $ W(\mfg,\mcG)^{tw} \subset ZX(\mfg)$ since $w(u) = z(-u-\ka/2) z(u-\ka/2)$ and $\wt{Y}(\mfg,\mcG)^{tw} \subset \wt{Y}(\mfg)$, where $\wt{Y}(\mfg)$ is the subalgebra of $X(\mfg)$ generated by the coefficients of $\mcT(u)$. ($\wt{Y}(\mfg)$ is isomorphic to the Yangian $Y(\mfg)$ - see \cite{AMR}.) Therefore, since $X(\mfg) \cong ZX(\mfg) \ot \wt{Y}(\mfg)$ \cite{AMR}, $X(\mfg,\mcG)^{tw}$ is isomorphic to $W(\mfg,\mcG)^{tw} \otimes \wt{Y}(\mfg,\mcG)^{tw}$. 

The kernel of the quotient homomorphism $X(\mfg,\mcG)^{tw} \onto Y(\mfg,\mcG)^{tw}$ is generated by $w_i, \, i\ge 1$. It follows from the decomposition $X(\mfg,\mcG)^{tw} \cong W(\mfg,\mcG)^{tw} \otimes \wt{Y}(\mfg,\mcG)^{tw}$ that $X(\mfg,\mcG)^{tw} \cong \mathrm{ker} \oplus \wt{Y}(\mfg,\mcG)^{tw}$ and thus $\wt{Y}(\mfg,\mcG)^{tw}$ is isomorphic to the image of the quotient homomorphism, that is, to $Y(\mfg,\mcG)^{tw}$.
\end{proof}

Let $f(u)$ be an invertible power series. The restriction of the map $\mu_f$ of $X(\mfg)$ (see \eqref{Aut}) to the subalgebra $X(\mfg,\mcG)^{tw}$ provides an automorphism of the latter; we denote it by $\nu_g$. Indeed, by \eqref{Aut} and \eqref{S=TGT} we have
$$
\mu_f : S(u) \mapsto f(u-\ka/2)\,T(u-\ka/2)\,\mcG(u)\,f(-u+\ka/2)\,T^{t}(-u+\ka/2) = f(u-\ka/2)f(-u+\ka/2)\,S(u) .
$$
From this we see that $g(u)$ given by $g(u)=f(u)f(-u)$ is an even series and $\nu_g(S(u)) = g(u-\ka/2)S(u)$.

\begin{crl} \label{C:32}
The algebra $\wt{Y}(\mfg,\mcG)^{tw}$ is stable under all automorphisms of the form $\nu_g$.
\end{crl}
\begin{proof}
We know already that $\mu_f(\mcT(u)) = \mcT(u)$, from which it follows that $\mu_f(\Sigma(u)) = \Sigma(u)$. The same holds for $\nu_g$, since it is obtained from $\mu_f$ by restriction.
\end{proof}

\begin{crl}\label{C:coideal2} 
The algebra $Y(\mfg,\mcG)^{tw}$ is a left coideal subalgebra of $Y(\mfg)$:
\eq{
\Delta(Y(\mfg,\mcG)^{tw}) \subset Y(\mfg) \ot Y(\mfg,\mcG)^{tw}. \nn
}
\end{crl}

\begin{proof}
This follows by Theorem \ref{T:Yt=Xt^Y} and analogous computation as in the proof of Proposition \ref{P:coideal}
\end{proof}

\begin{rmk} \label{matG}
The following observation will be useful in further sections. Let $A\in G$. The automorphism $\al_A$ of $X(\mfg)$ (see \eqref{Aut}) restricts to an isomorphism between $X(\mfg,\mcG)^{tw}$ and $X(\mfg,A^t\mcG A)^{tw}$ (see \eqref{S=TGT}). This isomorphism descends to the quotients $Y(\mfg,\mcG)^{tw}$ and $Y(\mfg,A^t\mcG A)^{tw}$. 
\end{rmk}

\subsection{Poincar\'e--Birkhoff--Witt Theorem for twisted Yangians}\label{Sec:32}

We first formulate this theorem in terms of the associated graded algebra of a certain filtration on the twisted Yangian and then in terms of a vector space basis. We first prove it for $\wt{Y}(\mfg,\mcG)^{tw}$ (and hence for $Y(\mfg,\mcG)^{tw}$ by Theorem \ref{T:Yt=Xt^Y}) and then for the extended twisted Yangian $X(\mfg,\mcG)^{tw}$.

Since $\wt{Y}(\mfg,\mcG)^{tw}$ is a subalgebra of $X(\mfg,\mcG)^{tw}$, it inherits its filtration, which in turn comes from the filtration on $X(\mfg)$ obtained by setting $\mathrm{deg} \, t_{ij}^{(m)}=m-1$. As a subalgebra of $X(\mfg)$, $\wt{Y}(\mfg)$ also inherits a filtration.

The following lemma will be useful in the proof of Proposition \ref{P:Yt=Ug[x]^q} and Corollary \ref{X:PBW} below.

\begin{lemma}
Denote respectively by $\bar t^{(m)}_{ij}$ and $\bar s^{(m)}_{ij}$ the images of $t^{(m)}_{ij}$ and $s^{(m)}_{ij}$ in the $(m-1)$-th homogeneous component of the associated graded algebra $\gr\, X(\mfg)$. Let $\bar {\tau}^{(m)}_{ij}$ and, respectively, $\bar{\si}^{(m)}_{ij}$ denote the images of $\tau^{(m)}_{ij}$ and $\si^{(m)}_{ij}$ in the $(m-1)$-st homogeneous component of $\gr\,\wt{Y}(\mfg)$. 
Then the following equalities hold:
\eqa{ \label{s^m}
\bar{s}_{ij}^{(m)} & = \mysum_{a=-n}^n \big(\bar t^{(m)}_{ia}g_{aj} + (-1)^m \theta_{ja} g_{ia}\bar t^{(m)}_{-j,-a}\big) + \delta_{m1} \bar g_{ij}, \\
 \label{sigma^m}
\bar{\sigma}_{ij}^{(m)} & = \mysum_{a=-n}^n \big(\bar {\tau}^{(m)}_{ia}g_{aj} + (-1)^m \theta_{ja} g_{ia}\bar \tau^{(m)}_{-j,-a}\big) + \delta_{m1} \bar g_{ij}. 
}
where $\bar g_{ij} = 0$ if $\mcG(u)$ is of the first kind and $\bar g_{ij} = (g_{ij}-\delta_{ij})c^{-1}$ if $\mcG(u)$ is of the second kind.
\end{lemma}

\begin{proof}
The proofs of both identities are very similar, so we consider only \eqref{s^m}. Let $g_{ij}(u)$ denote the matrix elements of $\mcG(u)$. Then the matrix elements of $S(u)$ are expressed as
\eq{ \label{s_ij(u)}
s_{ij}(u) = \mysum_{a,b=-n}^n \theta_{jb}\, t_{ia}(u-\ka/2)\,g_{ab}(u)\,t_{-j,-b}(-u+\ka/2).
}
We have
\eqn{
t_{ia}(u-\ka/2) &= \mysum_{r\ge0} t_{ia}^{(r)}(u-\ka/2)^{-r} 
= \delta_{ia} + \mysum_{r\ge1} \mysum_{s\ge0} t_{ia}^{(r)} {\scriptscriptstyle\small\begin{pmatrix} s+r-1 \\ s\end{pmatrix}} \left(\frac{\ka}{2}\right)^s u^{-s-r} ,
}
and
\eqn{
t_{-j,-b}(-u+\ka/2) &= \mysum_{r\ge0} (-1)^r t_{-j,-b}^{(r)}(u-\ka/2)^{-r} 
= \delta_{jb} + \mysum_{r\ge1} \mysum_{s\ge0} (-1)^r t_{-j,-b}^{(r)}\, {\scriptscriptstyle\small\begin{pmatrix} s+r-1 \\ s\end{pmatrix}} \left(\frac{\ka}{2}\right)^s u^{-s-r} . \nonumber
}
Set $f^{(r)}(u) = \mysum_{s\ge0} {\scriptscriptstyle\small\begin{pmatrix} s+r-1 \\ s\end{pmatrix}} {(\ka/2)}^s u^{-s}$ and $f^{(0)}(u)=1$. %
Let $\mcG(u)$ be of the first kind. Then $\mcG(u)=\mcG$ and $g_{ij}(u)=g_{ij}$, giving
\eqa{ \label{s_ij(u):1}
s_{ij}(u) = g_{ij} & + \mysum_{a=-n}^n \mysum_{r\ge1} \big(t^{(r)}_{ia} g_{aj} + (-1)^r \theta_{aj} g_{ia} t^{(r)}_{-j,-a} \big) f^{(r)}(u)\, u^{-r} \el
& + \mysum_{a,b=-n}^n \mysum_{r,s\ge1} (-1)^s \theta_{bj} t^{(r)}_{ia} g_{ab} t^{(s)}_{-j,-b} f^{(r)}(u)f^{(s)}(u)u^{-r-s} ,
}
and 
$$
\bar{s}_{ij}^{(m)} = \mysum_{a=-n}^n \big(\bar t^{(m)}_{ia}g_{aj} + (-1)^m \theta_{ja} g_{ia}\bar t^{(m)}_{-j,-a}\big) \text{ for } m\ge 1.
$$
Let $\mcG(u)$ of the second kind. Then 
$$
\mcG(u) = (I-c\,u\,\mcG)(1-c\,u)^{-1} = \mcG + (\mcG-I)\mysum_{t\ge1}c^{-t}u^{-t} ,
$$
and
\eqa{ \label{s_ij(u):2}
s_{ij}(u) = g_{ij} &+ g'_{ij}\mysum_{t\ge1}c^{-t}u^{-t} + \mysum_{a=-n}^n \mysum_{r\ge1} \big(t^{(r)}_{ia} g_{aj} + (-1)^r \theta_{aj} g_{ia} t^{(r)}_{-j,-a} \big) f^{(r)}(u)\, u^{-r} \el
& + \mysum_{a=-n}^n \mysum_{r,t\ge1} \big(t^{(r)}_{ia} g'_{aj} + (-1)^r \theta_{aj} g'_{ia} t^{(r)}_{-j,-a} \big) f^{(r)}(u)\, c^{-t}u^{-r-t} \el
& + \mysum_{a,b=-n}^n \mysum_{r,s\ge1} (-1)^s \theta_{bj} t^{(r)}_{ia} \left( g_{ab} + \mysum_{t\ge1}g'_{ab}c^{-t}u^{-t}\right) t^{(s)}_{-j,-b} f^{(r)}(u)f^{(s)}(u)u^{-r-s} , 
}
where $g'_{ab}=g_{ab}-\delta_{ab}$.
This time we get
\[
\bar{s}_{ij}^{(m)} = \mysum_{a=-n}^n \big(\bar t^{(m)}_{ia}g_{aj} + (-1)^m \theta_{ja} g_{ia}\bar t^{(m)}_{-j,-a}\big) + \delta_{m1} (g_{ij} - \delta_{ij}) c^{-1} \text{ for } m\ge 1. \qedhere
\]
\end{proof}

The twisted current algebra $\mfg[x]^\rho$ is defined as the subspace of $\mfg[x]$ consisting of elements fixed by the involution $\rho$ extended to $\mfg[x]$ by $\rho(F \ot p(x)) = \rho(F) \ot p(-x)$ for all $F\in\mfg$. The next result is an analogue for twisted Yangians of Theorem 3.6 in \cite{AMR}.

\begin{prop} \label{P:Yt=Ug[x]^q} 
The graded algebra $\gr\,\wt{Y}(\mfg,\mcG)^{tw}$ is isomorphic to the enveloping algebra $\mfU\mfg[x]^\rho$ of the twisted current algebra $\mfg[x]^\rho$.
\end{prop}

\begin{proof}
The Lie algebra $\mfg[x]^\rho$ is the linear span of the elements
\eq{
F_{ij}^{(\rho,m)} = \big( F_{ij} - (-1)^m \mcG F_{ij} \mcG^{-1} \big) x^{m-1} \qquad\text{with}\qquad -n\le i,j \le n, \; m\ge1. \nn
}
It is also spanned by the elements $F_{ij}^{\prime(\rho,m)} $ defined by \begin{equation}F_{ij}^{\prime(\rho,m)} = \mysum_{a=-n}^n \big( F_{ia} g_{aj} - (-1)^m g_{ia} F_{aj} \big) x^{m-1}. \label{Fprime} \end{equation}
Let us see why this is true. The $(a,b)$ entry of $\mcG F_{ij} \mcG^{-1}$ is $\mysum_{c,d=-n}^n g_{ac} (F_{ij})_{cd} g_{db}$ and $(F_{ij})_{cd} = (E_{ij} - \theta_{ij} E_{-j,-i})_{cd} = \delta_{ic} \delta_{jd} - \delta_{-j,c} \delta_{-i,d} \theta_{ij} $, so 
\eqn{
(a,b) \text{ entry of } \mcG F_{ij} \mcG^{-1} & = \mysum_{c,d=-n}^n g_{ac} (\delta_{ic} \delta_{jd} - \delta_{-j,c} \delta_{-i,d} \theta_{ij} ) g_{db} \\
& = \mysum_{c,d=-n}^n \delta_{ic} \delta_{jd} g_{ac} g_{db} - \mysum_{c,d} \delta_{-j,c} \delta_{-i,d} \theta_{ij} g_{ac} g_{db} \\
& = g_{ai} g_{jb} - \theta_{ij} g_{a,-j} g_{-i,b} .
}
Therefore, 
\eqn{
\mcG F_{ij} \mcG^{-1} & = \mysum_{a,b=-n}^n (g_{ai} g_{jb} - \theta_{ij} g_{a,-j} g_{-i,b}) E_{ab} \\
& = \mysum_{a,b=-n}^n g_{ai} E_{ab} g_{jb} - \mysum_{a,b=-n}^n \theta_{ij} g_{-a,-j} E_{-a,-b} g_{-i,-b} \\
& = \mysum_{a,b=-n}^n g_{ai} F_{ab} g_{jb} = \mysum_{a,b=-n}^n g_{ia} F_{ab} g_{bj} .
}
It follows that $\mysum_{b=-n}^n F_{ib}^{\prime(\rho,m)} g_{bj} = F_{ij}^{(\rho,m)}$, which shows that \[ \mathrm{span}_{\C} \{ F_{ij}^{(\rho,m)} \, | \, -n\le i,j \le n \} \subset \mathrm{span}_{\C} \{ F_{ij}^{\prime(\rho,m)} \, | \, -n\le i,j \le n \}. \] Indeed, \begin{align*} \mysum_{b=-n}^n F_{ib}^{\prime(\rho,m)} g_{bj} & = \mysum_{a,b=-n}^n (F_{ia} g_{ab} g_{bj} - (-1)^m g_{ia} F_{ab} g_{bj}) x^{m-1} \\
& = \mysum_{a=-n}^n \left( F_{ia} \delta_{aj} - (-1)^m \mysum_{b=-n}^n g_{ia} F_{ab} g_{bj}\right) x^{m-1} = F_{ij}^{(\rho,m)}. \end{align*} Moreover, \[ \mysum_{j=-n}^n (F_{ij} - (-1)^m \mcG F_{ij} \mcG^{-1} ) g_{jk} x^{m-1} = \mysum_{a,j=-n}^n F_{ia}^{\prime(\rho,m)} g_{aj} g_{jk} = \mysum_{a=-n}^n F_{ia}^{\prime(\rho,m)} \delta_{ak} = F_{ik}^{\prime(\rho,m)}, \] which implies that \[ \mathrm{span}_{\C} \{ F_{ij}^{\prime(\rho,m)} \, | \, -n\le i,j \le n \} \subset \mathrm{span}_{\C} \{ F_{ij}^{(\rho,m)} \, | \, -n\le i,j \le n \}, \] hence equality holds. It will be useful later to know that $F_{-j,-i}^{\prime(\rho,m)}= (\pm)\, \theta_{ij} (-1)^m F_{ij}^{\prime(\rho,m)}$.

 By (\cite{AMR}, Theorem 3.6) there exists an isomorphism $\psi : \mfU\mfg[x] \to \gr\,\wt{Y}(\mfg),\, F_{ij} x^{m-1} \mapsto \bar\tau_{ij}^{(m)}$. Using the symmetry of $\bar\tau_{ij}^{(m)}$ and \eqref{sigma^m}, we can write $$ \bar{\si}_{ij}^{(m)}=\mysum_{a=-n}^n \big(\bar\tau^{(m)}_{ia}g_{aj} - (-1)^m g_{ia}\bar\tau^{(m)}_{aj}\big) + \delta_{m1} \bar g_{ij}$$ and we have
\eq{
\psi : F_{ij}^{\prime(\rho,m)} \longmapsto \mysum_{a=-n}^n \Big( \bar\tau_{ia}^{(m)} g_{aj} - (-1)^m g_{ia} \bar\tau^{(m)}_{aj} \Big) = \bar{\si}_{ij}^{(m)} - \delta_{m1} \bar g_{ij}. \label{psi:F->si}
}
By \eqref{psi:F->si}, $\psi(\mfU\mfg[x]^\rho)$ is the subalgebra generated by the $\bar\si^{(m)}_{ij}$. We must show that it is all of $\gr\,\wt Y(\mfg,\mcG)^{tw}$.

The matrix $\Si(u)=q(u)^{-1}S(u)$ satisfies the reflection, symmetry and unitarity relations by Proposition \ref{sususcalar}, Lemmas \ref{L:RSRS}, \ref{L:Sym} and $q(u)=q(\ka-u)$. Here degree means total word degree, with $\deg\si^{(r)}_{ij}=r-1$. Set $V^{(m)}=(\si^{(m)}_{ij}-\delta_{m1}\bar g_{ij})$ and $Y=V^{(m)}\mcG$. Symmetry and unitarity give, modulo polynomials in the generators of lower word degree,
\[
Y^t=-Y,\qquad \mcG Y\mcG=(-1)^{m+1}Y.
\]
For $m=1$ these identities hold exactly. These are the defining linear conditions of the corresponding eigenspace of $\rho$ in $\mfg$. The matrix $(F^{(\rho,m)}_{ij})$ satisfies the same conditions and its entries span this eigenspace, so its linear relations are precisely the functionals vanishing on that eigenspace. Since $(F^{\prime(\rho,m)}_{ij})\mcG=(F^{(\rho,m)}_{ij})$, these relations lift to $V^{(m)}$.

The reflection equation makes commutators linear in top word degree. Induction on the word degree and the number of factors therefore gives spanning by ordered monomials in those entries of $V^{(m)}$ for which the $F^{\prime(\rho,m)}_{ij}$ form a basis of $\mfg[x]^\rho$. By \eqref{psi:F->si}, their leading terms are the images under $\psi$ of a Poincar\'e--Birkhoff--Witt basis. Hence no leading terms cancel, and $\gr\,\wt Y(\mfg,\mcG)^{tw}=\psi(\mfU\mfg[x]^\rho)$.
\end{proof}

\begin{crl} \label{C:33} 
Set $F_{ij}^{\prime \rho} = \mysum_{a=-n}^n \left( F_{ia}g_{aj} + g_{ia}F_{aj} \right)$. The assignment $F_{ij}^{\prime \rho} \mapsto {\si}_{ij}^{(1)} - \bar{g}_{ij}$ defines an embedding $\mfU\mfg^\rho \into \wt{Y}(\mfg,\mcG)^{tw}$. 
\end{crl}

\begin{rmk} \label{R:flat-def}
The algebra $Y(\mfg,\mcG)^{tw}$ may be considered as a flat deformation of the algebra $\mfU\mfg[x]^\rho$. Introduce a formal deformation parameter $\hbar$. Let $Y_{\hbar}(\mfg,\mcG)^{tw}$ be the $\C[\hbar]$-subalgebra of $Y(\mfg,\mcG)^{tw} \otimes_{\C} \C[\hbar]$ generated by $\tl s^{(r)}_{ij} = \hbar^{r-1}s_{ij}^{(r)}$ for $r\ge 1$. (Here we denote by $s_{ij}^{(r)}$ also the image of $s_{ij}^{(r)}$ under the quotient homomorphism $X(\mfg,\mcG)^{tw} \onto Y(\mfg,\mcG)^{tw}$.) For $a\in\C^{\times}$, $Y_{\hbar}(\mfg,\mcG)^{tw}/(\hbar-a)Y_{\hbar}(\mfg,\mcG)^{tw}$ is isomorphic to $Y(\mfg,\mcG)^{tw}$, whereas $Y_{\hbar}(\mfg,\mcG)^{tw}/ \hbar Y_{\hbar}(\mfg,\mcG)^{tw}$ is isomorphic to the enveloping algebra of $\mfg[x]^{\rho}$.
\end{rmk}

Now we formulate the Poincar\'e--Birkhoff--Witt property in terms of a vector space basis, which could be useful for obtaining a basis of a Verma module as in Section 4.2 in \cite{Mo3}. Suppose that $\mcG$ is a diagonal matrix. Then we have:
$$
\bar{\si}_{ij}^{(m)} = (g_{jj} - (-1)^m g_{ii})\,\bar\tau^{(m)}_{ij} + \delta_{m1}\bar g_{ij} .
$$
Let $\mcG$ be in the BDI case. We write (assuming that $g_{00}=1$ and $g_{-j,j}, g_{j,-j}$ do not appear if $j=0$)
$$ 
\bar{\si}_{ij}^{(m)}= \bar\tau^{(m)}_{ij}g_{jj} + \bar\tau^{(m)}_{i,-j}\,g_{-j,j} - (-1)^m g_{ii}\,\bar\tau^{(m)}_{ij} - (-1)^m g_{i,-i}\,\bar\tau^{(m)}_{-i,j} + \delta_{m1} \bar g_{ij}.
$$
Define \[ \mcQ^{\pm}=\left\{\pm\left(\tfrac{p-q-k}{2}+1\right),\ldots,\pm\left(\tfrac{N-k}{2}\right)\right\}, \quad \mcP=\left\{-\tfrac{p-q-k}{2}, \ldots, \tfrac{p-q-k}{2}\right\} = \mcP^{< 0} \cup \mcP^{> 0 } \text{ or } \mcP^{< 0} \cup \{ 0 \} \cup \mcP^{> 0 } \] where $k=0$ if $N$ is even and $k=1$ if $N$ is odd (i.e. $g_{ii}=1$, $g_{i,-i}=0$ for $i\in\mcP$ and $g_{ii}=0$, $g_{i,-i}=1$ for $i\in\mcQ^{\pm}$). 

Then
\eqa{
\bar{\si}_{ij}^{(m)} &= \bar\tau^{(m)}_{ij} - (-1)^m \bar\tau^{(m)}_{ij} \qquad \text{for} \qquad i,j\in\mcP , \label{BDI:a} \\
\bar{\si}_{ij}^{(m)} - \delta_{m1} \bar{g}_{ij} &= \bar\tau^{(m)}_{i,-j} - (-1)^m \bar\tau^{(m)}_{-i,j} \qquad \text{for} \qquad i,j\in\mcQ^{\pm} , \label{BDI:b}\\
\bar{\si}_{ij}^{(m)} &= \bar\tau^{(m)}_{ij} - (-1)^m \bar\tau^{(m)}_{-i,j} \qquad \text{for} \qquad i\in\mcQ^{\pm}, j\in\mcP, \label{BDI:c}\\
\bar{\si}_{ij}^{(m)} &= \bar\tau^{(m)}_{i,-j} -(-1)^m \bar\tau^{(m)}_{ij} \qquad \text{for} \qquad i\in\mcP, j\in\mcQ^{\pm}. \label{BDI:d}
}
Recall that a vector space basis of $\wt{Y}(\mfg)$ is provided by the ordered monomials in the generators $\tau_{ij}^{(r)}$ with $r\ge1$ and $i+j>0$ in the orthogonal case and $i+j\ge0$ in the symplectic case (\cite{AMR}, Corollary 3.7). 
This together with what was considered above implies the following analogue of the Poincar\'e--Birkhoff--Witt Theorem for the algebra $Y(\mfg,\mcG)^{tw}$.

\begin{thrm} \label{Y:PBW} 
Given any total ordering on the set of generators $\si_{ij}^{(r)}$ of $\wt{Y}(\mfg,\mcG)^{tw}$, a vector space basis of $\wt{Y}(\mfg,\mcG)^{tw}$ is provided by the ordered monomials in the following generators $(r\ge 1)$:
\vspace{-.4em}
\begin{itemize} [itemsep=0.1ex]
\item BD0: \, $\si_{ij}^{(2r-1)}$ with $i+j>0$.

\item C0: \;\;\; $\si_{ij}^{(2r-1)}$ with $i+j\ge0$.

\item CI: \;\;\;\; $\si_{ij}^{(2r-1)}$ with $i,j>0$; and $\si_{ij}^{(2r)}$ with $i+j\ge0$, $ij < 0$.

\item DIII: \; $\si_{ij}^{(2r-1)}$ with $i,j>0$; and $\si_{ij}^{(2r)}$ with $i+j>0$, $ij < 0$.

\item CII: \;\;\, $\si_{ij}^{(2r-1)}$ with $i+j\ge0$ and $|i|,|j|\le\frac{q}{2}$ or $|i|,|j|\ge \frac{q}{2}+1$; \\ and \;\;\;\, $\si_{ij}^{(2r)}$ \quad with $i\ge\frac{q}{2}+1$, $-\frac{q}{2}\le j\le \frac{q}{2}$ or $j\ge\frac{q}{2}+1$, $-\frac{q}{2}\le i\le \frac{q}{2}$.

\item BDI: \;\, $\si^{(2r-1)}_{ij}$ with $i+j>0$ and either $i,j\in\mcP$, or $i\in\mcP^{> 0}$, $j\in\mcQ^+$, or $i \in \mcQ^+$, $j\in\mcP^{> 0}$ or $i=0,j\in\mcQ^+$ (in type BI only); we should also include $\si^{(2r-1)}_{ij}$ when $i>|j|$ and $i,j\in\mcQ^{\pm}$; \\
and \;\;\;\, $\si_{ij}^{(2r)}$ \quad with $i+j>0$ and either $i\in\mcP^{> 0}$, $j\in\mcQ^+$ or $i\in\mcQ^+$, $j\in\mcP^{> 0}$ or $i=0,j\in\mcQ^+$ (in type BI only); we should also include $\si_{ij}^{(2r)}$ when $i\ge |j|, i\neq j$ and $i, j\in\mcQ^{\pm}$.
\end{itemize}
\end{thrm}

\begin{proof}
This follows from Proposition \ref{P:Yt=Ug[x]^q}. The Lie algebra $\mfg[x]^{\rho}$ is spanned by the matrices $F_{ij}^{\prime(\rho,m)} $ for $-n \le i,j \le n$, so all we need to do is to extract a basis from this spanning set. This will lead via the Poincar\'e--Birkhoff--Witt Theorem to a basis of $\mfU \mfg[x]^{\rho}$ which, by Proposition \ref{P:Yt=Ug[x]^q}, corresponds to a basis of $\wt{Y}(\mfg,\mcG)^{tw}$ consisting of ordered monomials in some of the generators $\si_{ij}^{(r)}$.

Let's explain a little bit how the basis is obtained in the BDI case. By considering the four cases $i,j\in\mcP$, $i,j\in\mcQ^\pm$, $i\in\mcP$, $j\in\mcQ^\pm$ and $i\in\mcQ^\pm$, $j\in\mcP$ and using the anti-symmetry $F_{ij} = -F_{-j,-i}$, it can be checked that $\mfg^{\rho}\ot \C x^{2r}$ is spanned by $F_{ij}^{\prime(\rho,2r+1)}$ with $i+j>0$ and either $i,j\in\mcP$, or $i\in\mcP^{> 0}$, $j\in\mcQ^+$, or $i\in\mcQ^+$, $j\in\mcP^{> 0}$, or $i=0,j\in\mcQ^+$ (in type BI only), or $i,j\in\mcQ^\pm$ with $i> |j|$. All these elements are linearly independent and there are exactly $\frac{p(p-1) + q(q-1)}{2}$ of them, which is the dimension of $\mfg^{\rho}$. Indeed, there are $\frac{(p-q)^2-(p-q)}{2}$ with $i,j\in\mcP$, $q(p-q-k)$ with $i\in\mcP^{> 0}$, $j\in\mcQ^+$ or $i\in\mcQ^+$, $j\in\mcP^{> 0}$, and there are $q^2-q$ with $i,j\in\mcQ^{\pm}$ and $i> |j|$ (plus $q$ elements with $i=0$, $j\in\mcQ^+$ in type BI).

As for $\check{\mfg}^{\rho}\ot \C x^{2r+1}$, it is spanned by $F_{ij}^{\prime(\rho,2r+2)}$ with $i+j>0$ and either $i\in\mcP^{> 0}$, $j\in\mcQ^+$, or $i\in\mcQ^+$, $j\in\mcP^{> 0}$, or $i=0,j\in\mcQ^+$ (in type BI only), or $i,j\in\mcQ^\pm$ with $i\ge |j|, i\neq j$. These are all linearly independent and there are $pq$ such elements, which is the dimension of $\check{\mfg}^{\rho}$.
\end{proof}

\begin{crl} \label{X:PBW} 
Given any total ordering on the set of generators $s_{ij}^{(r)}$ of $X(\mfg,\mcG)^{tw}$, a vector space basis of $X(\mfg,\mcG)^{tw}$ is provided by the ordered monomials in the generators $w_i$ with $i=2,4,6,\ldots$, and $s^{(r)}_{ij}$ with $r,i,j$ satisfying the same constraints as in Theorem \ref{Y:PBW}.
\end{crl}

\begin{proof}
By Proposition \ref{P:Yt=Ug[x]^q}, Theorem \ref{T:Yt=Xt^Y} and Corollary \ref{C:ZX(g,G)} ({\it 3}) (to be proved below), the graded algebra $\gr\,X(\mfg,\mcG)^{tw}$ is isomorphic to the tensor product of $\mfU\mfg[x]^\rho$ and the polynomial algebra $\C[\xi_2,\xi_4,\ldots]$ in the indeterminates $\xi_i$. Denote by $\bar w_m$ the images of the elements $w_m$ in the $(m-1)$-th homogeneous component of $\gr\,X(\mfg,\mcG)^{tw}$. Recall that
\eq{
\bar{t}_{ij}^{(m)} + \theta_{ij}\bar{t}_{-j,-i}^{(m)} = \delta_{ij} \bar{z}_m \qquad\text{and}\qquad
\bar\tau_{ij}^{(m)} = \frac{1}{2} (\bar{t}_{ij}^{(m)} - \theta_{ij} \bar{t}_{-j,-i}^{(m)}) . \nn
}
Moreover, we have $\bar{w}_m = 2 \bar{z}_m$. Then, by \eqref{s^m}, we see that the image of the element $s^{(m)}_{ij}$ in the $(m-1)$-th component of $\gr\,X(\mfg,\mcG)^{tw}$ is given by 
\eq{ \label{s=si+w}
\bar{s}^{(m)}_{ij} = \bar{\si}^{(m)}_{ij} + \frac{1}{4}(1+(-1)^m)g_{ij}\bar w_m .
}
Hence, by \eqref{psi:F->si}, we obtain an isomorphism
\eq{
F_{ij}^{\prime(\rho,m)} \mapsto \ol{s}_{ij}^{(m)} - \frac{1}{4}(1+(-1)^m) g_{ij} \ol{w}_m - \delta_{m1} \bar{g}_{ij}, \nn
}
with $\bar w_m$ being the image of $\xi_m$. This concludes the proof.
\end{proof}

\subsection{The centre of twisted Yangians} \label{Sec:34}

In order to deduce that $W(\mfg,\mcG)^{tw}$ is the centre of $X(\mfg,\mcG)^{tw}$, we will need to know that the centre of $\wt{Y}(\mfg,\mcG)^{tw}$ is trivial. By the previous proposition, this reduces to the same problem for the enveloping algebra of the twisted current algebra.

\begin{prop} \label{P:C=0} 
The centre of the enveloping algebra of $\mfg[x]^{\rho}$ is trivial.
\end{prop}

\begin{proof}
Since $\mfg$ is a semi-simple Lie algebra, $\mfg^{\rho}$ is reductive in $\mfg$ and it is equal to its own normalizer in $\mfg$ by Proposition 1.13.3 in \cite{Di}. It follows that if a non-zero element of $\mfg$ is invariant under the adjoint action of $\mfg^{\rho}$, then this element belongs to the normalizer of $\mfg^{\rho}$, hence it belongs to the centre of $\mfg^{\rho}$. Consequently, if the centre of $\mfg^{\rho}$ is trivial, then $\mfg$ does not have any nonzero elements invariant under the adjoint action of $\mfg^{\rho}$ and Lemma 2.8.1 in \cite{Mo3} allows us to conclude that the enveloping algebra of $\mfg[x]^{\rho}$ has no centre. The results cited in this proof only need $\mfg$ to be semi-simple, so the case $\mfg=\mfso_4$ is included.

The remaining cases are CI, DIII and BDI with $q=2$ or $p=2$. Except for DI with $p=q=2$, which we treat below, the centre of $\mfg^\rho$ is $\C J$ and $\mathrm{ad}\,J$ is invertible on $\check\mfg^\rho$. We use only these two properties. We may choose $J$ as follows; the square identities hold on $\check\mfg^\rho$:
\[
\begin{aligned}
\text{CI, DIII:}\quad &J=\textstyle\sum_{i=1}^n F_{ii},\qquad (\mathrm{ad}\,J)^2=4\,\id,\\
\text{BDI, }q=2:\quad &J=F_{n,n-1}+F_{-n,1-n}-F_{n,1-n}-F_{-n,n-1},\qquad (\mathrm{ad}\,J)^2=-4\,\id,\\
(N,p,q)=(3,2,1):\quad &J=F_{-1,0}-F_{0,-1},\qquad (\mathrm{ad}\,J)^2=-2\,\id.
\end{aligned}
\]
Here $\pm(n-1),\pm n$ are the indices in $\mcQ^\pm$ in the BDI case. In types CI and DIII we have $[F_{\mp i,\pm j},J]=\pm2F_{\mp i,\pm j}$ for $i,j\ge1$.

Let $C$ be an element in the centre of $\mfU\mfg[x]^{\rho}$ which is not a scalar. A basis of $\mfU\mfg[x]^{\rho}$ is provided by ordered monomials (in some fixed chosen order) of the elements $Xx^{2m}$ and $Yx^{2m+1}$, $m\ge0$, where $X$ runs over a basis of $\mfg^{\rho}$ and $Y$ over a basis of $\check{\mfg}^{\rho}$. Therefore, we can write $C$ as a sum of such monomials.

Since $C$ is in the centre of $\mfU\mfg[x]^{\rho}$, $[Jx^{2r},C]=0$ for all $r\ge0$ and it follows that the monomials in $C$ cannot include any $Yx^{2m+1}$ with $Y\in\check{\mfg}^{\rho}$, $m\ge0$, because $\mathrm{ad}\,J$ is invertible on $\check{\mfg}^{\rho}$. (This can be seen by taking $r$ larger than any of the exponents of $x$ appearing in any of the monomials which add up to $C$.) Therefore, $C$ is a central element in $\mfU\mfg^{\rho}[x^2]$ and it follows from Lemma 1.7.4 in \cite{Mo3} that the centre of $\mfU\mfg^{\rho}[x^2]$ is a polynomial ring in the variables $Jx^{2m}$, $m\ge0$. However, such an element cannot be in the centre of $\mfU\mfg[x]^{\rho}$ unless it is a scalar, again because $\mathrm{ad}\,J$ is invertible on $\check{\mfg}^{\rho}$.

It remains to treat $\mfg=\mfso_4$ with $p=q=2$, where the centre of $\mfg^{\rho}=\mfso_2\op\mfso_2$ is two-dimensional. Here $\mcG\in SO_4(\C)$, so $\rho$ preserves both simple ideals of $\mfso_4\cong\mfsl_2\op\mfsl_2$ and acts on each as the involution of the pair $(\mfso_3,\mfso_2)$. So $\mfU\mfg[x]^{\rho}$ is the tensor product of two copies of the algebra treated above, and its centre is the tensor product of their centres, which is trivial.
\end{proof}

\smallskip

The next corollary is the analogue of Corollary 3.9 in \cite{AMR} for the (extended) twisted Yangians.

\begin{crl} \label{C:ZX(g,G)}
The following statements hold: 
\vspace{-.4em}
\begin{enumerate} [itemsep=-0.1ex]
\item The centre of the algebra $Y(\mfg,\mcG)^{tw}$ is trivial.
\item The coefficients of the even series $w(u)$ generate the whole centre $ZX(\mfg,\mcG)^{tw}$ of $X(\mfg,\mcG)^{tw}$.
\item The coefficients $w_{2i}$ of the even series $w(u)$ are algebraically independent, so the subalgebra $ZX(\mfg,\mcG)^{tw}$ of $X(\mfg,\mcG)^{tw}$ is isomorphic to the algebra of polynomials in countably many variables.
\end{enumerate}
\end{crl}

\begin{proof}
{\it 1} follows Propositions \ref{P:Yt=Ug[x]^q} and \ref{P:C=0}. By Theorem \ref{T:Yt=Xt^Y}, the even coefficients of the series $w(u)$ generate the whole centre of $X(\mfg,\mcG)^{tw}$, which proves {\it 2}.

As for {\it 3}, the algebraic independence of the $w_{2i}$ is a consequence of the algebraic independence of the central elements $z_i$ and the fact that, since $w(u) = z(-u-\kappa/2) z(u-\kappa/2)$, we have that $w_{2i} = 2z_{2i} + p(z_1,\ldots, z_{2i-1})$ for some polynomial $p(z_1,\ldots, z_{2i-1})$.
\end{proof}

\begin{crl} \label{C:Xt=ZXt*Yt} 
The algebra $X(\mfg,\mcG)^{tw}$ is isomorphic to the tensor product of its centre and the subalgebra $\wt{Y}(\mfg,\mcG)^{tw}$.
\end{crl}

\begin{proof}
This is an immediate consequence of Theorem \ref{T:Yt=Xt^Y} and Corollary \ref{C:ZX(g,G)}.
\end{proof}

\subsection{Quantization of a left Lie coideal structure}

It was shown in (\cite{AACFR}, Theorem 3.3) that $\wt Y(\mfg)$ is a homogeneous quantization of a Lie bi-algebra $(\mfg[x],\bdel)$ (see Definition 6.2.6 in \cite{CP}), where $\bdel$ is a cobracket on $\mfg[x]$ defined as follows. $\mfg[x]$ is equal to ${\rm span}_\C\{F^{(r)}_{ij}\,|-n\leq i,j\leq n,\, r\geq 1 \}$, where $F^{(r)}_{ij} = F_{ij} \,x^{r-1}$ and the grading on $\mfg[x]$ is given by $\deg F^{(r)}_{ij} = r-1$. (For convenience we set $F^{(0)}_{ij}=0$.) Then
\eq{
\bdel(F^{(r)}_{ij})= \mysum_{a=-n}^{n} \mysum_{s=1}^{r-1} \left( F^{(r-s)}_{ia} \ot F^{(s)}_{aj} - F^{(s)}_{aj} \ot F^{(r-s)}_{ia} \right) . \label{cobrack1} 
}
Set $\wt \mcT(u) = (\mcT(u/\hbar)-I)/\hbar$ and let $\tl\tau_{ij}(u)$ denote the matrix elements of $\wt\mcT(u)$. Then, for the coefficient $\tl{\tau}^{(r)}_{ij}$ of $u^{-r}$ in $\tl\tau_{ij}(u)$, set
$$ 
\bdel(\tl{\tau}^{(r)}_{ij}) = \frac{1}{\hbar}\left( \Delta(\tl{\tau}^{(r)}_{ij}) - \Delta^{op}(\tl{\tau}^{(r)}_{ij}) \right).
$$
Since $\Delta(\tl{\tau}^{(r)}_{ij})= \tl{\tau}^{(r)}_{ij}\ot1 + 1\ot\tl{\tau}^{(r)}_{ij} + \hbar \displaystyle{\mysum_{a=-n}^n \mysum_{s=1}^{r-1}} \left( \tl{\tau}^{(r-s)}_{ia} \ot \tl{\tau}^{(s)}_{aj}\right)$, it follows that 
\eq{
\bdel(\tl{\tau}^{(r)}_{ij}) = \mysuml_{a=-n}^n \mysuml_{s=1}^{r-1} \left( \tl{\tau}^{(r-s)}_{ia} \ot \tl{\tau}^{(s)}_{aj} - \tl{\tau}^{(s)}_{aj} \ot \tl{\tau}^{(r-s)}_{ia} \right). \label{cobrack2}
}
Using the generators $\tl{\tau}_{ij}^{(r)}$, we can define a flat deformation $\wt Y_{\hbar}(\mfg)$ of $\mfU\mfg[x]$ and $\wt Y_{\hbar}(\mfg)/\hbar \wt Y_{\hbar}(\mfg) \cong \mfU\mfg[x]$ by identifying $\tl{\tau}_{ij}^{(r)} \,({\rm mod}\;\hbar)$ with $\bar{\tau}_{ij}^{(r)}$. (See Remark \ref{R:flat-def} where $Y_\hbar(\mfg,\mcG)^{tw}$ ($\cong\wt Y_\hbar(\mfg,\mcG)^{tw}$) is considered.) Upon this identification, the cobracket \eqref{cobrack2} becomes \eqref{cobrack1}.

We want to show that an analogous result holds for $\wt{Y}_\hbar(\mfg,\mcG)^{tw}$. It will be convenient for us to use the language of twisted Manin triples and left Lie coideals as introduced by S. Belliard and N. Crampe in (\cite{BeCr}, Definitions 2.1 and 2.2). (What we call left Lie coideals are termed left Lie bi-ideals in \textit{loc.~cit.}) Left Lie coideal structures for twisted current algebras were constructed by one of the current authors in (\cite{BeRe}, Section 4.2). Let $\mfg=\mfg^\rho\op\check\mfg^\rho$ be the symmetric pair decomposition of $\mfg$ with respect to the involution $\rho$. The standard Lie bialgebra structure on the current algebra $\mfg[x]$ comes from the Manin triple $(\mfg((x^{-1})),\mfg[x],x^{-1}\mfg[[x^{-1}]])$ with the non-degenerate ad-invariant bilinear form $\langle \cdot, \cdot \rangle$ given by $\langle F_1 x^r, F_2 x^{-s} \rangle = \kappa(F_1,F_2) \delta(r=s-1)$ where $\kappa(F_1,F_2) = \frac{1}{2}\mathrm{Tr}(F_1F_2)$. (In particular, $\kappa(F_{ij},F_{ji}) =1$ if $j\neq -i$ and $\kappa(F_{ij},F_{kl}) =0$ if $(k,l) \neq (j,i)$ and $(k,l)\neq (-i,-j)$; if $j=-i$, then $F_{i,-i}=0$ in the orthogonal case and $\kappa(F_{i,-i},F_{-i,i})=2$ in the symplectic case.) The involution $\rho$ can be naturally extended to $\mfg((x^{-1}))$ and we will denote its extension also by $\rho$. The pairing $\langle \cdot, \cdot \rangle$ is not $\rho$-invariant because 
\[ 
\langle \rho(F_1 x^r), \rho(F_2 x^{-r-1}) \rangle = - \langle \rho(F_1) x^r, \rho(F_2) x^{-r-1} \rangle = - \kappa(F_1,F_2) = - \langle F_1 x^r, F_2 x^{-r-1} \rangle. 
\] 
The map $\rho$ can thus be viewed as an anti-invariant Manin triple twist in the terminology of \cite{BeCr}.

The cobracket $\bdel$ associated to this Manin triple is $\rho$-anti-invariant in the sense that $\bdel(\rho(F x^r)) = - \left(\rho \ot \rho\right) (\bdel(F x^r))$: 
\eqn{ 
\langle \bdel(\rho(F_1 x^{r_1})), (F_2 x^{r_2}) \ot (F_3 x^{r_3}) \rangle & = \langle \rho(F_1 x^{r_1}), [F_2 x^{r_2}, F_3 x^{r_3}] \rangle \\
& = - \langle F_1 x^{r_1}, \rho([F_2 x^{r_2}, F_3 x^{r_3}]) \rangle \\
& = - \langle F_1 x^{r_1}, [\rho(F_2 x^{r_2}), \rho(F_3 x^{r_3})] \rangle = - \langle \bdel(F_1 x^{r_1}), \rho(F_2 x^{r_2}) \ot \rho(F_3 x^{r_3}) \rangle. 
}
It follows that $\bdel(\mfg[x]^{\rho}) \subset (\check{\mfg}[x]^{\rho} \ot \mfg[x]^{\rho}) \op (\mfg[x]^{\rho} \ot \check{\mfg}[x]^{\rho})$. (Here $\check{\mfg}[x]^{\rho}$ is the eigenspace of $\mfg[x]$ for the eigenvalue $-1$ of $\rho$.) The restriction of $\bdel$ to $\mfg[x]^{\rho}$ can thus be decomposed as $\bdel = \btau + \btau'$ where $\btau$ is the composite of $\bdel$ with the projection onto $\check{\mfg}[x]^{\rho} \otimes \mfg[x]^{\rho}$ and similarly for $\btau'$. (Note that $\btau' = -\sigma\circ\btau$ where $\si : a\ot b \mapsto b\ot a$ is the flip operator. $\btau'$ yields a right Lie coideal structure.) 

The Lie algebra $\mfg[x]^{\rho}$ is spanned by the elements $F_{ij}^{\prime(\rho,m)} $ and using \eqref{Fprime} one can compute $\bdel(F_{ij}^{\prime(\rho,m)})$ and find that 
\eqn{
\bdel(F_{ij}^{\prime(\rho,r)}) & = \mysuml_{a=-n}^n \mysuml_{s=1}^{r-1} \big( F_{ia}^{(s)} \ot F_{aj}^{\prime(\rho,r-s)} - (-1)^{s} F_{aj}^{(s)} \ot F_{ia}^{\prime(\rho,r-s)} \\
& \qquad \qquad \qquad - F_{aj}^{\prime(\rho,r-s)} \ot F_{ia}^{(s)} + (-1)^s F_{ia}^{\prime(\rho,r-s)} \ot F_{aj}^{(s)} \big) .
}
Let $\pi_-$ be the projection of $\mfg[x]$ onto $\check\mfg[x]^\rho$ for the extended involution $\rho$. It follows that
\eq{
\btau(F_{ij}^{\prime(\rho,r)}) = (\pi_-\ot\id)\mysuml_{a=-n}^n \mysuml_{s=1}^{r-1} \big( F_{ia}^{(s)} \ot F_{aj}^{\prime(\rho,r-s)} - (-1)^{s} F_{aj}^{(s)} \ot F_{ia}^{\prime(\rho,r-s)} \big). \label{btauF}
}

\begin{rmk}\label{R:sym}
Let $\mfk=\mfg[x]^\rho$, $f=F_{ij}^{\prime(\rho,r)}$, and let $A$ be the unprojected sum in \eqref{btauF}. Since $A\in\mfg[x]\ot\mfk$ and $\bdel(f)=A-\si(A)$ has no component in $\mfk\ot\mfk$, we have $A=\btau(f)+B$ with $B\in\mfk\ot\mfk$ symmetric. For diagonal $\mcG$, $B=0$ if $f\neq0$; in type CI with $N=2$, however, $F_{11}^{\prime(\rho,2)}=0$ and $A=4F_{11}\ot F_{11}$. In type BDI, $B$ can be non-zero. In condition (4) of Definition 5.3 in \cite{BeRe}, changing a lift $X$ to $X+\hbar Z$ adds the symmetric element $\Delta(\bar Z)-\bar Z\ot1-1\ot\bar Z$ modulo $\hbar$. We read this condition modulo symmetric elements. For $X=\tl\sigma^{(r)}_{ij}$ as below, the lift $X-\frac{\hbar}{2}Z$ gives equality modulo $\hbar$, where $Z$ lifts the product $m(B)\in\mfU\mfk$, since $\Delta(m(B))-m(B)\ot1-1\ot m(B)=2B$.
\end{rmk}

\begin{thrm}\label{T:quan}
The algebra $\wt{Y}(\mfg,\mcG)^{tw}$ is a homogeneous quantization of the left Lie coideal $(\mfU\mfg[x]^\rho,\btau)$, with condition (4) of Definition 5.3 in \cite{BeRe} read as in Remark \ref{R:sym}.
\end{thrm}

\begin{proof}
We have to verify item $\mathit{(4)}$ in Definition 5.3 of \cite{BeRe}: for items $\mathit{(1)-(3)}$, see \cite[Theorem~3.3]{AACFR}, Remark \ref{R:flat-def} and Corollary \ref{C:coideal2}.
Set $\wt \Sigma(u)=(\Sigma(u/\hbar)-\mcG(u/\hbar))/\hbar$ and let $\tl \sigma_{ij}(u)$ denote the matrix elements of $\wt \Sigma(u)$. Then, for the coefficient $\tl{\sigma}^{(r)}_{ij}$ of $u^{-r}$ in $\tl{\sigma}_{ij}(u)$, set
$$
\wt{\btau}(\tl{\sigma}^{(r)}_{ij}) = \frac{1}{\hbar}\left( \Delta(\tl{\sigma}^{(r)}_{ij}) - \left(\phi(\tl{\sigma}^{(r)}_{ij})\ot1+1\ot\tl{\sigma}^{(r)}_{ij}\right) \right) \in Y_{\hbar}(\mfg) \ot \wt{Y}_\hbar(\mfg,\mcG)^{tw}[[u^{-1}]].
$$
In this proof, $\phi$ denotes the inclusion $\wt{Y}_{\hbar}(\mfg,\mcG)^{tw} \subset \wt{Y}_{\hbar}(\mfg)$. Using \eqref{cop:s(u)}, for $r\ge 1$ we find
\eqa{ \label{D(si):hbar}
\Delta(\tl{\sigma}^{(r)}_{ij}) 
&= \mysum_{a=-n}^n\left( \tl{\tau}^{(r)}_{ia}g_{aj} + (-1)^r \theta_{ja}\,g_{ia}\,\tl{\tau}^{(r)}_{-j,-a} \right) \ot 1 + 1\ot\tl{\sigma}^{(r)}_{ij} \el
& \quad + \hbar \,(r-1)\,\frac{\ka}{2} \mysum_{a=-n}^n\left( \tl{\tau}^{(r-1)}_{ia}g_{aj} - (-1)^{r} \theta_{ja}\,g_{ia}\,\tl{\tau}^{(r-1)}_{-j,-a} \right) \ot 1 \el
& \quad + \hbar \mysum_{a=-n}^n \big(\tl\tau^{(r-1)}_{ia} \bar{g}_{aj} - (-1)^{r} \theta_{aj} \bar{g}_{ia} \tl\tau^{(r-1)}_{-j,-a} \big) \ot 1 \el
& \quad + \hbar \mysum_{a,b=-n}^n \mysum_{s=1}^{r-1} \theta_{jb} (-1)^s \tl{\tau}^{(r-s)}_{ia} g_{ab} \tl{\tau}^{(s)}_{-j,-b} \ot 1 \el
& \quad + \hbar \mysum_{a=-n}^n \mysum_{s=1}^{r-1} \left( \tl{\tau}^{(s)}_{ia} \ot \tl{\sigma}^{(r-s)}_{aj} + (-1)^s\theta_{ja} \tl{\tau}^{(s)}_{-j,-a}\ot \tl{\sigma}^{(r-s)}_{ia}\right)  + \mcO(\hbar^2) ,
}
where $\mcO(\hbar^2)$ denotes elements of quadratic and higher order in $\hbar$. 
Then, by observing that, for $r\ge1$,
\eqn{
\phi(\tl{\sigma}^{(r)}_{ij}) &= \mysum_{a=-n}^n\left( \tl{\tau}^{(r)}_{ia}g_{aj} + (-1)^r \theta_{ja}\,g_{ia}\,\tl{\tau}^{(r)}_{-j,-a} \right) \\
& \quad + \hbar \,(r-1)\,\frac{\ka}{2} \mysum_{a=-n}^n\left( \tl{\tau}^{(r-1)}_{ia}g_{aj} - (-1)^{r} \theta_{ja}\,g_{ia}\,\tl{\tau}^{(r-1)}_{-j,-a} \right) \\
& \quad + \hbar \mysum_{a=-n}^n \big(\tl\tau^{(r-1)}_{ia} \bar{g}_{aj} - (-1)^{r} \theta_{aj} \bar{g}_{ia} \tl\tau^{(r-1)}_{-j,-a} \big) \\
& \quad + \hbar \mysum_{a,b=-n}^n \mysum_{s=1}^{r-1} \theta_{jb} (-1)^s \tl{\tau}^{(r-s)}_{ia} g_{ab} \tl{\tau}^{(s)}_{-j,-b} + \mcO(\hbar^2)
}
and using the symmetry relation $\tl{\tau}^{(r)}_{ij} \equiv -\theta_{ij}\,\tl{\tau}^{(r)}_{-j,-i} \;({\rm mod}\,\hbar)$, we obtain, for $r\ge 1$,
\eqn{
\wt{\btau}(\tl{\sigma}^{(r)}_{ij}) & = \mysum_{a=-n}^n \mysum_{s=1}^{r-1} \left( \tl{\tau}^{(s)}_{ia} \ot \tl{\sigma}^{(r-s)}_{aj} - (-1)^s \tl{\tau}^{(s)}_{aj}\ot \tl{\sigma}^{(r-s)}_{ia} \right) +\mcO(\hbar).
}
Since $\tl{\tau}^{(s)}_{ia} \equiv F_{ia}^{(s)} \; ({\rm mod}\, \hbar)$ and $\tl{\sigma}^{(r-s)}_{aj} \equiv F_{aj}^{\prime(\rho,r-s)} \; ({\rm mod}\, \hbar)$ (see \eqref{psi:F->si}; here $\tl\sigma^{(1)}_{ij}=\si^{(1)}_{ij}-\bar g_{ij}$, since $\wt\Sigma(u)$ subtracts $\mcG(u/\hbar)$), we can conclude from \eqref{btauF} that $\wt\btau \equiv \btau \; ({\rm mod}\, \hbar)$, up to the symmetric element $B$ of Remark \ref{R:sym}.
\end{proof}

\section{Reflection algebras} \label{S:ra}

In this section, we introduce a reflection algebra $\mcB(\mcG)$ defined via the $R$-matrix given by \eqref{R(u)}, the matrix $\mcG$ and an additional symmetry relation. We show that the extended twisted Yangian $X(\mfg,\mcG)^{tw}$ given by Definition \ref{D:X(g,G)} is isomorphic to the algebra $\mcB(\mcG)$, but first we prove a similar isomorphism for $Y(\mfg,\mcG)^{tw}$ and a quotient of $\mcB(\mcG)$. The usual notation $\pm$ and $\mp$ will distinguish orthogonal (upper sign) and symplectic (lower sign) cases. The lower sign in $(\pm)$ will distinguish the cases CI and DIII from the other cases. 

Using solutions of the reflection equation, quantum analogues of symmetric spaces were introduced in \cite{NoSu}. By analogy, we may think of twisted Yangians as affine and quantized versions of symmetric spaces.

\begin{defn}
The reflection algebra $\mcB(\mcG)$ is the unital associative algebra generated by elements $\mss_{ij}^{(r)}$ for $-n\le i,j\le n$, $r\in\Z_{\ge 0}$ satisfying the reflection equation 
\eq{ \label{RE}
R(u-v)\,\msS_1(u)\,R(u+v)\,\msS_2(v) = \msS_2(v)\,R(u+v)\,\msS_1(u)\,R(u-v) ,
}
and the symmetry relation
\eq{
\msS^t(u) = (\pm)\,\msS(\ka-u) \pm \frac{\msS(u)-\msS(\ka-u)}{2u-\ka} + \frac{\Tr(\mcG(u))\,\msS(\ka-u) - \Tr(\msS(u))\cdot I }{2u-2\ka} \,, \label{RES}
}
where the $S$-matrix $\msS(u)$ is defined by
\eq{ \label{XS(u)}
\msS(u) = \mysum_{i,j=-n}^n E_{ij}\ot \mss_{ij}(u)\in \End(\C^N)\ot \mcB(\mcG)[[u^{-1}]], 
\quad\; \mss_{ij}(u) = \mysum_{r=0}^{\infty} \mss_{ij}^{(r)}\,u^{-r} , 
\quad\; \mss^{(0)}_{ij} = g_{ij} . 
}
\end{defn}

The reflection equation \eqref{RE} is equivalent to the following set of relations:

\eqa{ \label{[s,s]}
[\,\mss_{ij}(u),\mss_{kl}(v)]&=\frac{1}{u-v}\Big(\mss_{kj}(u)\,\mss_{il}(v)-\mss_{kj}(v)\,\mss_{il}(u)\Big)\el
{}& +\frac{1}{u+v} \mysuml_{a=-n}^n \Big(\delta_{kj}\,\mss_{ia}(u)\,\mss_{al}(v)-
\delta_{il}\,\mss_{ka}(v)\,\mss_{aj}(u)\Big)\el
{}& -\frac{1}{u^2-v^2} \mysuml_{a=-n}^n \delta_{ij}\Big(\mss_{ka}(u)\,\mss_{al}(v) - \mss_{ka}(v)\,\mss_{al}(u)\Big) \el	
{}& -\frac{1}{u-v-\ka} \mysuml_{a=-n}^n \Big( \delta_{k,-i}\,\theta_{i a}\, \mss_{a j}(u)\, \mss_{-a,l}(v) - \delta_{l,-j} \,\theta_{a j}\, \mss_{k,-a}(v)\, \mss_{i a}(u) \Big) \el
{}& -\frac{1}{u+v-\ka}\, \Big( \theta_{j,-k}\,\mss_{i,-k}(u)\, \mss_{-j,l}(v) - \theta_{i,-l}\,\mss_{k,-i}(v)\, \mss_{-l,j}(u) \Big) \nn \el
{}& +\frac{1}{(u+v) (u-v-\ka)}\, \theta_{i,-j} \mysuml_{a=-n}^n \Big( \delta_{k,-i}\,\mss_{-j,a}(u)\, \mss_{a l}(v) - \delta_{l,-j}\,\mss_{k a}(v)\, \mss_{a,-i}(u) \Big) \el
{}& +\frac{1}{(u-v) (u+v-\ka)}\, \theta_{i,-j} \Big( \mss_{k,-i}(u)\, \mss_{-j,l}(v)-\mss_{k,-i}(v)\, \mss_{-j,l}(u) \Big) \el
{}& -\frac{1}{(u-v-\ka) (u+v-\ka)}\, \theta_{i j} \mysuml_{a=-n}^n \Big( \delta_{k,-i}\, \mss_{a a}(u)\, \mss_{-j,l}(v) - \delta_{l,-j}\,\mss_{k,-i}(v)\, \mss_{a a}(u) \Big) .
}

The symmetry relation \eqref{RES} is equivalent to
\eqa{ \label{sym}
\theta_{ij}\mss_{-j,-i}(u) = (\pm)\,\mss_{ij}(\ka-u) \pm \frac{\mss_{ij}(u)-\mss_{ij}(\ka-u)}{2u-\ka} + \frac{\Tr(\mcG(u))\,\mss_{ij}(\kappa-u) - \delta_{ij} \mysum_{k=-n}^n \mss_{kk}(u) }{2u-2\ka} .
}

\begin{rmk}
Let us comment on the choice of the reflection equation \eqref{RE} for the algebra $\mcB(\mcG)$. Consider the twisted reflection equation
\eq{ \label{tRE}
R(u-v)\,{\msS}'_1(u)\,R^t(-u-v)\,{\msS}'_2(v) = {\msS}'_2(v)\,R^t(-u-v)\,{\msS}'_1(u)\,R(u-v) .
}
Observe that $R^t(u)=R(\ka-u)$. Then it is possible to see that \eqref{tRE} is equivalent to \eqref{RE} upon identification ${\msS}'(u)=\msS(u+\ka/2)$. Moreover, the choice of \eqref{RE} has motivated the form of the $S$-matrix $S(u)$ in \eqref{S=TGT}. For the twisted reflection equation \eqref{tRE} the natural choice would be ${S}'(u)=T(u)\,\mcG(u+\ka/2)\,T^t(-u)$, the unitarity relation would become ${S}'(u)\,{S}'(-\ka-u)=I$.
\end{rmk}

\begin{rmk} \label{R:new-feat}
The set of relations \eqref{[s,s]} and the symmetry relation \eqref{sym} are the analogues of those obtained in \cite{MNO}. In particular, the first three lines of \eqref{[s,s]} coincide with (3.7.2) in \textit{loc.\@ cit} and the first three terms of \eqref{sym} with $\kappa=0$ and the plus sign in $(\pm)$ coincide with (3.6.4) in \textit{loc.\@ cit}. The additional terms are essentially new features of the twisted Yangians of types B,C,D.
\end{rmk}

Consider an arbitrary even power series $g(u)\in1+u^{-2}\C[[u^{-2}]]$ and a matrix $A\in G$ such that $A\mcG A^t = \mcG$. The maps
\eq{
\nu_g \;:\; \msS(u) \mapsto g(u-\ka/2)\,\msS(u) \quad\text{and}\quad \al_A\;:\;\msS(u) \mapsto A \msS(u) A^t,\label{nu_g} 
}
are automorphisms of $\mcB(\mcG)$, as can be seen from the symmetry relation \eqref{RES}. Furthermore, the map $\msS(u) \mapsto (\pm)\,\msS^{t}(u)$ is an anti-automorphism of $\mcB(\mcG)$ (in cases CI and DIII, $\mcG^t=-\mcG$, so the sign is needed to keep the constant term $\mcG$ of $\msS(u)$). This is verified by taking the transpose of the reflection equation \eqref{RE} and using the transpose symmetry of the $R$-matrix, $R^{t_1 t_2}(u)=R(u)$. The compatibility with the symmetry relation \eqref{RES} is straightforward.

By dropping the symmetry relations one obtains an extended reflection algebra $\mcX\mcB(\mcG)$. We will consider this extension in Section \ref{Sec:5}. The following lemmas will be needed to establish a homomorphism from the algebra $\mcB(\mcG)$ to the extended twisted Yangian $X(\mfg,\mcG)^{tw}$.

\subsection{Solutions of the reflection and symmetry equations}

\begin{lemma} \label{L:RPRP} 
The matrix $\mcG(u)$ is a solution of the reflection equation
\eq{ \label{P-RE}
R_{12}(u-v)\,\mcG_1(u)\,R_{12}(u+v)\,\mcG_2(v) = \mcG_2(v)\,R_{12}(u+v)\,\mcG_1(u)\,R_{12}(u-v) .
}
\end{lemma}

\begin{proof} 
Let $\mcG(u)$ be of the first kind. Then it is enough to prove the following equalities:
\eq{ \label{L31:A1}
\left(1 - \frac{P}{u-v} \right) \mcG_1 \left( 1 - \frac{P}{u+v} \right) \mcG_2 = \mcG_2 \left(1 - \frac{P}{u+v} \right) \mcG_1 \left(1 - \frac{P}{u-v} \right) 
}
and
\vspace{-.5em}
%
%
%
\eqa{
& \qquad\qquad P \mcG_1 Q \mcG_2 = \mcG_2 Q \mcG_1 P , \quad
Q \mcG_1 P \mcG_2 = \mcG_2 P \mcG_1 Q , \label{L31:A3} \\
& \mcG_1 Q \mcG_2 = \mcG_2 Q \mcG_1 , \quad
Q \mcG_1 \mcG_2 = \mcG_2 \mcG_1 Q , \quad
Q \mcG_1 Q \mcG_2 = \mcG_2 Q \mcG_1 Q . \qquad\label{L31:A6}
}
$\circ$ \eqref{L31:A1} can be expanded and checked directly using $P \mcG_1 = \mcG_2 P$, $P \mcG_2 = \mcG_1 P$ and $\mcG^2=I$. \vspace{0.25em}\\
$\circ$ \eqref{L31:A3} follows by similar arguments as \eqref{L31:A1} and $PQ=QP$. \vspace{0.25em}\\
$\circ$ The first identity in \eqref{L31:A6} can be checked directly:
$$
\mcG_1 Q \mcG_2 = \mysum_{i,j=-n}^n \theta_{ij} g_{ii} g_{-j,-j} E_{ij} \ot E_{-i,-j} = \mysum_{i,j=-n}^n \theta_{ij} g_{-i,-i} g_{jj} E_{ij} \ot E_{-i,-j} = \mcG_2 Q \mcG_1 .
$$
$\circ$ The second identity in \eqref{L31:A6} is
$$
Q \mcG_1 \mcG_2 = \mysum_{i,j=-n}^n \theta_{ij} g_{jj} g_{-j,-j} E_{ij} \ot E_{-i,-j} = (\pm)\,Q = \mysum_{i,j=-n}^n \theta_{ij} g_{ii} g_{-i,-i} E_{ij} \ot E_{-i,-j} = \mcG_2 \mcG_1 Q.
$$
$\circ$ For the third identity in \eqref{L31:A6}, we have
\eqa{
Q \mcG_1 Q \mcG_2 & = \left( \mysum_{i,j=-n}^n \theta_{ij} g_{jj} E_{ij} \ot E_{-i,-j} \right) \left( \mysum_{k,l=-n}^n \theta_{kl} g_{-l,-l} E_{kl} \ot E_{-k,-l} \right) \el
& = \mysum_{i,l=-n}^n \left( \mysum_{j=-n}^n \theta_{ij} \theta_{jl} g_{jj} \right) g_{-l,-l} E_{il} \ot E_{-i,-l} \,, \label{QGQG}
}
and similarly
\eqa{
\mcG_2 Q \mcG_1 Q 
= \mysum_{i,l=-n}^n \left( \mysum_{j=-n}^n \theta_{ij} \theta_{jl} g_{jj} \right) g_{-i,-i} E_{il} \ot E_{-i,-l} \,. \label{GQGQ}
}
Since $\theta_{ij} \theta_{jl} = \theta_{il}$, the sum $\mysum_{j=-n}^n \theta_{ij} \theta_{jl} g_{jj}$ vanishes when $\mcG$ is traceless, which is true in cases CI, DIII and also in cases DI, CII when $p=q$, so $Q \mcG_1 Q \mcG_2 = 0 = \mcG_2 Q \mcG_1 Q$ in those cases. If $\mcG=I$, we can see that $Q \mcG_1 Q \mcG_2 = \mcG_2 Q \mcG_1 Q$ is true also. These computations assume that $\mcG$ is diagonal, but in type DI with $p=q$ it is not. For any $\mcG$, the identities \eqref{L31:A6} follow from $A_1Q=A^t_2Q$, $QA_1=QA^t_2$ and $QA_1Q=\Tr(A)\,Q$, which hold for every $A\in\End\,\C^N$, and from $\mcG^t=(\pm)\,\mcG$, $\mcG^2=I$, and $\Tr(\mcG)=0$ when $\mcG\neq I$.

Now let $\mcG(u)$ be of the second kind, namely $\mcG(u) = (I-c\,u\,\mcG)(1-c\,u)^{-1}$ with $c=\frac{4}{p-q}$ and $p>q$. Then it is enough to prove the following equalities:
\eq{ 
\left(1 - \frac{P}{u-v} \right) \mcG_1(u) \left( 1 - \frac{P}{u+v} \right) \mcG_2(v) = \mcG_2(v) \left(1 - \frac{P}{u+v} \right) \mcG_1(u) \left(1 - \frac{P}{u-v} \right) , \label{L31:B1} 
}
and
{\allowdisplaybreaks
\eqa{
\mcG_2 Q P + Q P \mcG_2 &= \mcG_2 P Q + P Q \mcG_2 , \label{L31:B2} \\
P \mcG_1 Q \mcG_2 + Q \mcG_1 P \mcG_2 &= \mcG_2 P \mcG_1 Q + \mcG_2 Q \mcG_1 P, \label{L31:B3} \\
P \mcG_1 Q \mcG_2 + \mcG_2 P \mcG_1 Q &= \mcG_2 Q \mcG_1 P + Q \mcG_1 P \mcG_2 , \label{L31:B4}\\
2 Q \mcG_1 P + \mcG_2 P Q + \mcG_2 Q P &= 2 P \mcG_1 Q + P Q \mcG_2 + Q P \mcG_2 , \label{L31:B5}\\
P Q \mcG_2 + Q P \mcG_2 + \mcG_2 Q Q &= \mcG_2 P Q + \mcG_2 Q P + 2 \kappa (\mcG_2 Q-Q \mcG_2) + Q Q \mcG_2, \label{L31:B6}\\
\mcG_2 \mcG_1 Q + \mcG_1 Q \mcG_2 &= \mcG_2 Q \mcG_1 + Q \mcG_1 \mcG_2 , \label{L31:B7}\\
\mcG_2 \mcG_1 Q + \mcG_2 Q \mcG_1 &= \mcG_1 Q \mcG_2 + Q \mcG_1 \mcG_2 , \label{L31:B8}\\
2 Q(\mcG_1 + \mcG_2) - 2(\mcG_1 + \mcG_2) Q &= c(\mcG_2 P \mcG_1 Q+ \mcG_2 Q \mcG_1 P + Q \mcG_1 Q \mcG_2 \el
& \qquad - P \mcG_1 Q \mcG_2 - Q \mcG_1 P \mcG_2 - \mcG_2 Q \mcG_1 Q) . \label{L31:B9}
}
}
$\circ$ \eqref{L31:B1} can be expanded and checked directly using $P \mcG_1 = \mcG_2 P$, $P \mcG_2 = \mcG_1 P$. \vspace{0.25em}\\
$\circ$ \eqref{L31:B2}-\eqref{L31:B5} follow by $P \mcG_1 = \mcG_2 P$, $P \mcG_2 = \mcG_1 P$, $QP=PQ$ and $\mcG^2=I$. \vspace{0.25em}\\
$\circ$ By similar arguments \eqref{L31:B6} is equivalent to
$$
2P Q \mcG_2 - 2 \mcG_2 P Q = (N - 2 \kappa) (Q \mcG_2 - \mcG_2 Q) .
$$
Recall that $QP=PQ=\pm Q$ and $\ka=N/2\mp1$. Thus $2PQ=\pm2Q$ and $N-2\ka=\pm2$, and the equality holds

\noindent $\circ$ \eqref{L31:B7} and \eqref{L31:B8} are essentially the same and can be checked directly. They are true if $\mcG$ is the diagonal matrix in type CII. This follows by \eqref{L31:A6}. We only need to show that they are true for the BDI case. Recall that $\wt{\mcG} = C^{-1} \mcG C$ is a diagonal matrix and observe that $C_1^{-1} C_2^{-1} Q C_2 C_1 = \wt{Q}$ where $\wt{Q}= P^{'_2} = P^{'_1}$. ($A'$ denotes the transpose of the matrix $A$ with respect to the main diagonal; the index in $P^{'_1}$ indicates on which copy of $\End(\C^N)$ the transpose is taken.) Indeed, using that $K= C C'$ and $Q=P^{t_2} = K_2 P^{'_2} K_2$, we obtain 
\eqa{
C_1^{-1} C_2^{-1} Q C_2 C_1 & = C_1^{-1} C_2^{-1} K_2 P^{'_2} K_2 C_2 C_1 = C_1^{-1} C_2^{-1} C_2 C_2' P^{'_2} C_2 C_2' C_2 C_1 \el
& = C_1^{-1} C_2' P^{'_2} (C_2^{-1})^{'_2} C_1 = (C_1^{-1} C_2^{-1} P C_2 C_1)^{'_2} = \wt{Q} . \nonumber
}
Conjugating \eqref{L31:B7} by $C_1^{-1}$ and $C_2^{-1}$, we deduce that it is equivalent to $\wt{\mcG}_2 \wt{\mcG}_1 \wt{Q} + \wt{\mcG}_1 \wt{Q} \wt{\mcG}_2 = \wt{\mcG}_2 \wt{Q} \wt{\mcG}_1 + \wt{Q} \wt{\mcG}_1 \wt{\mcG}_2$, which can be checked as \eqref{L31:A6} since $\wt{\mcG}$ is diagonal with entries equal to $\pm 1$. The same argument gives \eqref{L31:B8}.

\vspace{0.25em}

\noindent $\circ$ The last equality, \eqref{L31:B9}, by similar arguments as before, is equivalent to
$$
2 Q(\mcG_1 + \mcG_2) - 2(\mcG_1 + \mcG_2) Q = c(Q \mcG_1 Q \mcG_2 - \mcG_2 Q \mcG_1 Q) .
$$
Conjugating by $C_1^{-1}$ and $C_2^{-1}$, we deduce that \eqref{L31:B9} is equivalent to 
$$
2 \wt{Q}(\wt{\mcG}_1 + \wt{\mcG}_2) - 2(\wt{\mcG}_1 + \wt{\mcG}_2) \wt{Q} = c(\wt{Q} \wt{\mcG}_1 \wt{Q} \wt{\mcG}_2 - \wt{\mcG}_2 \wt{Q} \wt{\mcG}_1 \wt{Q}) .
$$
$\wt{\mcG}$ is diagonal with entries $\wt{g}_{ii}=\pm 1$. Then
\eq{ \label{QGG+GGQ}
2 \wt{Q}(\wt{\mcG}_1 + \wt{\mcG}_2) - 2(\wt{\mcG}_1 + \wt{\mcG}_2) \wt{Q} = 4 \wt{Q} \wt{\mcG}_1 - 4 \wt{\mcG}_1 \wt{Q} = 4 \mysum_{i,j=-n}^n \,(\wt{g}_{jj}-\wt{g}_{ii})\, E_{ij}\ot E_{ij},
}
Similarly to \eqref{QGQG} and \eqref{GQGQ}, we have
\eq{
c( \wt{Q} \wt{\mcG}_1 \wt{Q} \wt{\mcG}_2 - \wt{\mcG}_2 \wt{Q} \wt{\mcG}_1 \wt{Q}) = c \left( \mysum_{k=-n}^n \wt{g}_{kk} \right) \mysum_{i,j=-n}^n (\wt{g}_{jj}-\wt{g}_{ii}) E_{ij} \ot E_{ij} \,, \label{QGQG-GQGQ} 
}
and $\mysum_{k=-n}^n \wt{g}_{kk} = p - q$, so the equality holds if $c=\frac{4}{p-q}$. 
\end{proof} 

\begin{lemma} \label{L:RSRS} 
The $S$-matrix $S(u)$ satisfies the reflection equation 
\eq{
R(u-v)\, S_1(u)\, R(u+v)\, S_2(v) = S_2(v)\, R(u+v)\, S_1(u)\, R(u-v) . \label{RSRS}
}
\end{lemma}

\begin{proof}
The proof of \eqref{RSRS} follows the standard method, see e.g.\ Section 3 of \cite{MNO}. We will need the following auxiliary relations:
\eqa{ 
T^t_1(-u+\ka/2)\, R(u+v)\, T_2(v-\ka/2) &= T_2(v-\ka/2)\, R(u+v)\, T^t_1(-u+\ka/2) , \label{L32:1} \\
T_1(u-\ka/2)\, R(u+v)\, T^t_2(-v+\ka/2) &= T^t_2(-v+\ka/2)\, R(u+v)\, T_1(u-\ka/2) , \label{L32:2} \\
R(u-v)\, T^t_1(-u+\ka/2)\, T^t_2(-v+\ka/2) &= T^t_2(-v+\ka/2)\, T^t_1(-u+\ka/2) \, R(u-v) . \label{L32:3} 
}
Let us show why these are true. The first relation is obtained by transposing the first factor of the ternary relation \eqref{RTT} and using symmetry of the $R$-matrix $R^t(u)=R(\ka-u)$,
$$
R(u-v)\, T_1(u)\, T_2(v) = T_2(v)\, T_1(u)\, R(u-v) \Longrightarrow T^t_1(u)\, R(\ka-u+v)\, T_2(v) = T_2(v)\, R(\ka-u+v)\, T^t_1(u).
$$
Then by substituting $u\to -u+\ka/2$ and $v\to v-\ka/2$ we obtain \eqref{L32:1}. The second relation \eqref{L32:2} is obtained from \eqref{L32:1} by conjugating with $P$,
$$
T^t_2(-u+\ka/2)\, R(u+v)\, T_1(v-\ka/2) = T_1(v-\ka/2)\, R(u+v)\, T^t_2(-u+\ka/2) ,
$$
and interchanging $u\leftrightarrow v$. The last relation \eqref{L32:3} is obtained by transposing the first factor of \eqref{L32:2}, giving 
$$
R(\kappa-u-v)\, T^t_1(u-\ka/2)\, T^t_2(-v+\ka/2) = T^t_2(-v+\ka/2)\, T^t_1(u-\ka/2) \, R(\kappa-u-v) ,
$$
and substituting $u\to-u+\ka$. Now 
\eqa{
R(u&-v)\, S_1(u)\, R(u+v)\, S_2(v) \el
& = R(u-v)\, T_1(u-\ka/2)\, \mcG_1(u)\, (T^t_1(-u+\ka/2)\, R(u+v)\, T_2(v-\ka/2))\, \mcG_2(v)\, T^t_2(-v+\ka/2) \el
& = (R(u-v)\, T_1(u-\ka/2)\, T_2(v-\ka/2))\, \mcG_1(u)\, R(u+v)\, T^t_1(-u+\ka/2)\, \mcG_2(v)\, T^t_2(-v+\ka/2) \quad\text{by \eqref{L32:1}} \el
& = T_2(v-\ka/2)\, T_1(u-\ka/2)\, (R(u-v)\, \mcG_1(u)\, R(u+v)\, \mcG_2(v))\, T^t_1(-u+\ka/2)\, T^t_2(-v+\ka/2) \quad\text{by \eqref{RTT}} \el
& = T_2(v-\ka/2)\, T_1(u-\ka/2)\, \mcG_2(v)\, R(u+v)\, \mcG_1(u)\, (R(u-v)\, T^t_1(-u+\ka/2)\, T^t_2(-v+\ka/2)) \quad\text{by \eqref{P-RE}} \el 
& = T_2(v-\ka/2)\, \mcG_2(v)\, (T_1(u-\ka/2)\, R(u+v)\, T^t_2(-v+\ka/2))\, \mcG_1(u)\, T^t_1(-u+\ka/2)\, R(u-v) \quad\text{by \eqref{L32:3}} \el 
& = T_2(v-\ka/2)\, \mcG_2(v)\, T^t_2(-v+\ka/2)\, R(u+v)\, T_1(u-\ka/2)\, \mcG_1(u)\, T^t_1(-u+\ka/2)\, R(u-v) \qquad\!\text{by \eqref{L32:2}} \el
& = S_2(v)\, R(u+v)\, S_1(u)\, R(u-v) . \nonumber \qedhere
}
\end{proof}

\begin{lemma} \label{L:Sym} 
The $S$-matrix $S(u)$ satisfies the symmetry relation
\eq{
S^t(u) = (\pm)\,S(\ka-u) \pm \frac{S(u)-S(\ka-u)}{2u-\ka} + \frac{\Tr(\mcG(u))S(\kappa-u) - \Tr(S(u))\cdot I}{2u-2\ka} \,. \label{S-Sym}
}
\end{lemma}

\begin{proof}
By \eqref{S=TGT} we have
\eqa{
(S^t(u))_{ij} &= \theta_{ij} s_{-j,-i}(u) = \mysum_{a,b=-n}^n \theta_{ij}\,\theta_{b,-i}\,g_{ab}(u)\, t_{-j,a}(u-\ka/2)\,t_{i,-b}(-u+\ka/2) , \nonumber
}
where $g_{ab}(u)$ denotes the matrix elements of $\mcG(u)$. Using commutation relations \eqref{[t,t]} we find
\eqa{
t_{-j,a}&(u-\ka/2)\,t_{i,-b}(-u+\ka/2) \el
& = t_{i,-b}(-u+\ka/2)\,t_{-j,a}(u-\ka/2) \el
& + \frac{1}{2u-\ka}\big( t_{ia}(u-\ka/2)\,t_{-j,-b}(-u+\ka/2) - t_{ia}(-u+\ka/2)\,t_{-j,-b}(u-\ka/2)\big) \el
& + \frac{1}{2u-2\ka} \mysuml_{c=-n}^n \big(\delta_{ab}\,\theta_{ac}\,t_{i,-c}(-u+\ka/2)\,t_{-j,c}(u-\ka/2) - \delta_{ij}\,\theta_{-j,c}\,t_{ca}(u-\ka/2)\,t_{-c,-b}(-u+\ka/2) \big) . \nonumber
}
First suppose that $\mcG(u)$ is of the first kind and $\mcG$ is diagonal. This excludes only type DI with $p=q$. Then $g_{ab}(u)=\delta_{ab}\,g_{aa}$ and $g_{aa}=(\pm)\,g_{-a,-a}$. In this case we find
\eqa{
(S^t(u))_{ij} &= \mysuml_{a=-n}^n \theta_{j,-a}\,g_{aa}\,t_{i,-a}(-u+\ka/2)\,t_{-j,a}(u-\ka/2) \el
& + \frac{1}{2u-\ka} \mysuml_{a=-n}^n \theta_{j,-a}\,g_{aa} \left( t_{ia}(u-\ka/2)\,t_{-j,-a}(-u+\ka/2) - t_{ia}(-u+\ka/2)\,t_{-j,-a}(u-\ka/2)\right) \el
& + \frac{1}{2u-2\ka} \mysuml_{a,b=-n}^n g_{aa}\left(\theta_{j,-b}\,t_{i,-b}(-u+\ka/2)\,t_{-j,b}(u-\ka/2) - \delta_{ij}\,\theta_{ab}\,t_{ba}(u-\ka/2)\,t_{-b,-a}(-u+\ka/2) \right) \el 
& = (\pm)(S(\ka-u))_{ij} \pm \frac{(S(u))_{ij} - (S(\ka-u))_{ij} }{2u-\ka} + \frac{\Tr(\mcG(u))\,(S(\ka-u))_{ij} - \delta_{ij}\,\Tr(S(u)) }{2u-2\ka} , \label{L33:1}
}
where in the second equality we have used the property $\theta_{j,-a}=\pm\theta_{ja}$ and the fact that $\mcG(u)=\mcG$ and $\Tr(\mcG)=0$ for all of the first kind cases except BCD0.

Now suppose that $\mcG(u)$ is of the second kind, or that $\mcG(u)=\mcG$ in type DI with $p=q$. In both cases we have $g_{ab}(u)=g_{-b,-a}(u)$ and $\theta_{ab}\,g_{bc}(u)=\theta_{ac}\,g_{bc}(u)$, so
\eqa{
(S^t(u))_{ij} &= \mysuml_{a,b=-n}^n \theta_{j,-b}\,g_{ab}(u)\,t_{i,-b}(-u+\ka/2)\,t_{-j,a}(u-\ka/2) \el
& + \frac{1}{2u-\ka}\mysuml_{a,b=-n}^n \theta_{j,-b}\,g_{ab}(u) \left( t_{ia}(u-\ka/2)\,t_{-j,-b}(-u+\ka/2) - t_{ia}(-u+\ka/2)\,t_{-j,-b}(u-\ka/2)\right) \el
& + \frac{1}{2u-2\ka} \mysuml_{a,b,c=-n}^n g_{ab}(u) \, ( \delta_{ab}\,\theta_{j,-c}\,t_{i,-c}(-u+\ka/2)\,t_{-j,c}(u-\ka/2) \el[-.25em]
& \hspace{8cm}- \delta_{ij}\,\theta_{bc}\,t_{ca}(u-\ka/2)\,t_{-c,-b}(-u+\ka/2) ) \el
& = \left(T(-u+\ka/2)\,\mcG(u)\,T^t(u-\ka/2)\right)_{ij} \pm \frac{1}{2u-\ka}\left(S(u)-T(-u+\ka/2)\,\mcG(u)\,T^t(u-\ka/2)\right)_{ij} \el
& + \frac{1}{2u-2\ka} \left( \Tr(\mcG(u))\,(T(-u+\ka/2)\,T^t(u-\ka/2))_{ij} - \delta_{ij} \Tr(S(u)) \right) . \label{L33:2}
}
In type DI with $p=q$, we have $\Tr(\mcG)=0$ and $\mcG(\ka-u)=\mcG$, so \eqref{L33:2} gives \eqref{S-Sym} directly. It remains to treat the second kind. Observe that $\Tr(\mcG(u))=(2\ka\pm2-4u)(1-cu)^{-1}$ and use the relation
$$
\mcG(u)=\mcG(\ka-u)+\frac{c\,(2u-\ka)(I-\mcG)}{(1-cu)(1-c\,(\ka-u))} \,.
$$
This gives the identity
$$
\left(1\mp\frac{1}{2u-\ka}\right)\mcG(u) + \frac{\Tr(\mcG(u))}{2u-2\ka} = \left(1\mp \frac{1}{2u-\kappa} + \frac{\Tr(\mcG(u))}{2u-2\kappa} \right) \mcG(\kappa-u) \,,
$$
that applied to \eqref{L33:2} gives 
$$
(S^t(u))_{ij} = (S(\ka-u))_{ij} \pm \frac{\left(S(u)-S(\ka-u)\right)_{ij}}{2u-\ka} + \frac{\Tr(\mcG(u))(S(\kappa-u))_{ij} - \delta_{ij}\,\Tr(S(u)) }{2u-2\ka} \,,
$$
which, combined with \eqref{L33:1}, proves the symmetry relation \eqref{S-Sym} in all cases.
\end{proof}

\subsection{Isomorphism between twisted Yangians and reflection algebras}

It follows by Lemmas \ref{L:RSRS} and \ref{L:Sym} that the following formula does indeed define a homomorphism $\phi: \mcB(\mcG) \to X(\mfg,\mcG)^{tw}$ which is surjective:
\eq{ \label{phi}
\phi \;:\; \mcB(\mcG) \to X(\mfg,\mcG)^{tw}, \quad \msS(u) \mapsto S(u)=T(u-\ka/2)\mcG(u)T^t(-u+\ka/2) ,
}
It will be proved in Theorem \ref{T:BGdecomp} below that $\phi$ is also injective. We will first establish an isomorphism between their quotients by the ideals generated by the coefficients of $\msw(u)-1$ and of $w(u)-1$, respectively. To achieve this, we start with the next proposition which is suggested by the analogous result (Proposition \ref{sususcalar}) for the extended twisted Yangian and will be needed to identify the quotient of $\mcB(\mcG)$ isomorphic to the twisted Yangian $Y(\mfg,\mcG)^{tw}$.

\begin{prop}\label{ssw}
In the algebra $\mcB(\mcG)$ the product $\msS(u)\,\msS(-u)=\msw(u)\cdot I$ is a scalar matrix, where $\msw(u)$ is an even formal power series in $u^{-1}$ with coefficients $\msw_i$ ($i=2,4,\ldots$) central in $\mcB(\mcG)$.
\end{prop}

\begin{proof}
By multiplying both sides of \eqref{[s,s]} by $(u^2-v^2)$ and setting $v=-u$, we have
\begin{align}
\delta_{ij} \mysum_{a=-n}^n \Big(\mss_{ka}(u)\, \mss_{al}(-u) & -\mss_{ka}(-u)\,\mss_{al}(u)\Big) \notag \\
& = 2 u \mysum_{a=-n}^n \Big( \delta_{jk}\,\mss_{ia}(u)\,\mss_{al}(-u) - \delta_{il}\,\mss_{ka}(-u)\,\mss_{aj}(u) \Big) \notag \\
& + \frac{2u}{2u-\ka}\,\theta_{i,-j} \mysum_{a=-n}^n \Big( \delta_{k,-i}\,\mss_{-j,a}(u)\,\mss_{al}(-u) - \delta_{l,-j}\,\mss_{ka}(-u)\,\mss_{a,-i}(u)\Big) . \label{susu1}
\end{align}
Suppose first that $k\neq \pm l$. Set $i=j=k$ in \eqref{susu1} to conclude that 
$$
\mysum_{a=-n}^n \Big( \mss_{ka}(u)\, \mss_{al}(-u) - \mss_{ka}(-u)\, \mss_{al}(u) \Big) = 2u \mysum_{a=-n}^n \mss_{ka}(u)\, \mss_{al}(-u).
$$
Setting $i=j=l$ in \eqref{susu1} gives 
\eq{
\mysum_{a=-n}^n \Big( \mss_{ka}(u)\, \mss_{al}(-u) - \mss_{ka}(-u)\, \mss_{al}(u) \Big) = -2u \mysum_{a=-n}^n \mss_{ka}(-u)\, \mss_{al}(u), \label{susu2} 
}
so $\mysum_{a=-n}^n \mss_{ka}(u) \mss_{al}(-u) = - \mysum_{a=-n}^n \mss_{ka}(-u) \mss_{al}(u)$. Now let us set $i=j=-k$ in \eqref{susu1} to obtain
\eq{
\mysum_{a=-n}^n \Big( \mss_{ka}(u)\, \mss_{al}(-u) - \mss_{ka}(-u)\, \mss_{al}(u) \Big) = \pm\frac{2u}{2u-\kappa} \mysum_{a=-n}^n \mss_{ka}(u)\, \mss_{al}(-u). \label{susu3}
}
The display before \eqref{susu2} and \eqref{susu3} imply that 
$$
\pm\frac{2u}{2u-\kappa} \mysum_{a=-n}^n \mss_{ka}(u)\, \mss_{al}(-u) =2u \mysum_{a=-n}^n\mss_{ka}(u) \,\mss_{al}(-u)
$$
and it follows that $\mysum_{a=-n}^n \mss_{ka}(u)\, \mss_{al}(-u) = 0 = \mysum_{a=-n}^n \mss_{ka}(-u)\, \mss_{al}(u)$.
This needs $k,l\neq0$. If $k=0$ (type B only), set $i=j=l$ and $i=j=-l$ in \eqref{susu1}. Then the left-hand side of \eqref{susu2} equals both $-2u\mysum_{a=-n}^n \mss_{0a}(-u)\, \mss_{al}(u)$ and $-\frac{2u}{2u-\kappa}\mysum_{a=-n}^n \mss_{0a}(-u)\, \mss_{al}(u)$. So this sum is zero, and then so is $\mysum_{a=-n}^n \mss_{0a}(u)\, \mss_{al}(-u)$. If $l=0$, set $i=j=k$ and $i=j=-k$ instead.

If $j=k=l$ and $i=-l$ in \eqref{susu1}, then $\mysum_{a=-n}^n \mss_{-l,a}(u)\, \mss_{al}(-u) = 0 $, and if $i=k=l, j=-k$, we obtain that $ \mysum_{a=-n}^n \mss_{k,a}(-u)\, \mss_{a,-k}(u) = 0 $.

We have showed that $\msS(u)\msS(-u)$ and $\msS(-u)\msS(u)$ are diagonal matrices. If $k=l$, then setting $i=j=k$ in \eqref{susu1} shows directly that $\mysum_{a=-n}^n \Big( \mss_{ka}(u)\, \mss_{al}(-u) - \mss_{ka}(-u)\,\mss_{al}(u) \Big) = 0$, so $\msS(u)\,\msS(-u) = \msS(-u)\,\msS(u)$.

If $N\ge 3$, then we can choose $i,j,k,l$ such that $i=l,j=k$ and $i\neq -j$, in which case we find from \eqref{susu1} that $\mysum_{a=-n}^n \mss_{ia}(u)\, \mss_{ai}(-u) = \mysum_{a=-n}^n \mss_{ja}(-u)\, \mss_{aj}(u) $. If $i=l = -j=-k$, then we get also $\mysum_{a=-n}^n \Big( \mss_{ia}(u)\, \mss_{ai}(-u) - \mss_{-i,a}(-u)\, \mss_{a,-i}(u) \Big)=0 $. Therefore, the diagonal entries of $\msS(u)\msS(-u)$ are all equal. 

The conclusion so far is that $\msw(u)$ is an even series. Showing that $\msw(u)$ is central in $\mcB(\mcG)$ is exactly as in Proposition 2.1 in \cite{MoRa}.
\end{proof}

\begin{defn}
Let $\mcU\mcB(\mcG)$ be the quotient of the reflection algebra $\mcB(\mcG)$ by the ideal generated by the entries of $\msS(u)\,\msS(-u) - I$. We will call the relation
\eq{ \label{Y:symm}
\msS(u)\,\msS(-u) = I . \vspace{-.75ex}
}
the unitarity constraint.
\end{defn}

\begin{thrm} \label{T:Y(g,G)}
The twisted Yangian $Y(\mfg,\mcG)^{tw}$ is isomorphic to $\mcU\mcB(\mcG)$.
\end{thrm}

\begin{proof}
The homomorphism $\phi$ descends to $\hat{\phi}:\mcU\mcB(\mcG) \lra Y(\mfg,\mcG)^{tw}$ and is surjective. We have to see why it is injective. This will be a consequence of the Poincar\'e--Birkhoff--Witt Theorem for $ Y(\mfg,\mcG)^{tw}$. Let us also denote by $\mss_{ij}^{(m)}$ the images of these generators in the quotient $\mcU\mcB(\mcG)$. We have a filtration on $\mcU\mcB(\mcG)$ obtained by assigning degree $m-1$ to $\mss_{ij}^{(m)}$ and $\hat{\phi}$ becomes a filtered homomorphism. 

Let $\bar\mss^{(m)}_{ij}$ denote the image of the abstract generator $\mss^{(m)}_{ij}$ in the $(m-1)$-th homogeneous component of $\gr\,\mcU\mcB(\mcG)$. The symmetry relation \eqref{RES} leads to the following relation in the $(m-1)$-th homogeneous component of the graded algebra: \eq{ \label{s=s}
\theta_{ij}\bar{\mss}_{-j,-i}^{(m)} = (\pm) (-1)^m \bar{\mss}_{ij}^{(m)} + \delta_{m1}\, \mathrm{tr}(\mcG) (g_{ij} - \delta_{ij})/2.
}
The defining relation \eqref{[s,s]} implies that the following relation holds in $\gr\,\mcU\mcB(\mcG)$:
\eqa{
[\bar{\mss}_{ij}^{(m_1)},\bar{\mss}_{kl}^{(m_2)}] &= g_{kj}\,\bar\mss_{il}^{(m_1+m_2-1)} - g_{il}\,\bar\mss_{kj}^{(m_1+m_2-1)} - (-1)^{m_1} \mysuml_{a=-n}^n \Big(\delta_{jk}\,g_{ia}\,\bar\mss_{al}^{(m_1+m_2-1)} - \delta_{il}\,g_{aj}\,\bar\mss_{ka}^{(m_1+m_2-1)} \Big) \el
& - \mysuml_{a=-n}^n \Big( \delta_{k,-i} \theta_{ia} \, g_{aj}\,\bar\mss_{-a,l}^{(m_1+m_2-1)} - \delta_{l,-j} \theta_{aj} \,g_{ia}\,\bar\mss_{k,-a}^{(m_1+m_2-1)} \Big) \el
& + (-1)^{m_1} \Big( \theta_{j,-k} g_{i,-k}\,\bar\mss_{-j,l}^{(m_1+m_2-1)} - \theta_{i,-l} g_{-l,j}\,\bar\mss_{k,-i}^{(m_1+m_2-1)}\Big). \label{[mss,mss]}
}
It can be checked directly that exactly the same relation holds for the generators $F^{\prime(\rho,m)}_{ij}$ of the Lie algebra $\mfg[x]^\rho$. The equalities \eqref{s=s} and \eqref{[mss,mss]}, together with the relation
\[
\mysum_{a=-n}^n \big( (\bar{\mss}_{ia}^{(m)} - \delta_{m1}\bar g_{ia})\, g_{aj} + (-1)^m g_{ia}\, (\bar{\mss}_{aj}^{(m)} - \delta_{m1}\bar g_{aj}) \big) = 0 ,
\]
imply that the elements $\bar{\mss}_{ij}^{(m)}$ satisfy all the defining relations of the generators $F^{\prime(\rho,m)}_{ij}$ of $\mfg[x]^\rho$, so there exists a surjective algebra homomorphism $\psi: \mfU\mfg[x]^\rho \onto \gr\,\mcU\mcB(\mcG)$ given by $F^{\prime(\rho,m)}_{ij} \mapsto \bar{\mss}_{ij}^{(m)} - \delta_{m1} \bar{g}_{ij}$. The displayed relation is the image, in the $(m-1)$-th homogeneous component, of the coefficient of $u^{-m}$ in the unitarity constraint \eqref{Y:symm}; the products $\mss^{(a)}_{ik}\,\mss^{(b)}_{kj}$ with $a,b\ge1$ and $a+b=m$ have degree at most $m-2$. Together with \eqref{s=s}, it gives all linear relations among the $F^{\prime(\rho,m)}_{ij}$ (see the proof of Proposition \ref{P:Yt=Ug[x]^q}). (The relation \eqref{[mss,mss]} also holds when $m=1$ and $\bar{\mss}_{ij}^{(1)}$ is replaced by $\bar{\mss}_{ij}^{(1)} - \bar{g}_{ij}$.) The composite of this homomorphism with $\gr \, \hat{\phi}: \gr\,\mcU\mcB(\mcG) \onto \gr\,Y(\mfg,\mcG)^{tw}$ is an isomorphism by Proposition \ref{P:Yt=Ug[x]^q}. Therefore, $\psi$ and $\gr \, \hat{\phi}$ must also be isomorphisms, and it follows that $\phi$ is an isomorphism. 
\end{proof}

Since $\phi(\msw(u)) = w(u) = q(u)q(-u)$ and the coefficients of $q(u)$ can be expressed in terms of the coefficients of $w(u)$ (see the proof of Theorem \ref{T:Yt=Xt^Y}), we can write $\msw(u) = \msq(u) \msq(-u)$ with $\phi(\msq(u)) = q(u)$. The central elements $w_{2i}, i=1,2,\ldots$ are algebraically independent by Corollary \ref{C:ZX(g,G)} and $\phi(\msw_{2i}) = w_{2i}$, so the central elements $\msw_{2i}, i=1,2,\ldots$ are also algebraically independent and $\phi$ provides an isomorphism between the subalgebra $\mcW\mcB(\mcG)$ of $\mcB(\mcG)$ generated by the elements $\msw_{2i}$ and the centre $ZX(\mfg,\mcG)^{tw}$ of $X(\mfg,\mcG)^{tw}$ according to Corollary \ref{C:ZX(g,G)}.

 Let $\wt{\mcU\mcB}(\mcG)$ denote the subalgebra of $\mcB(\mcG)$ generated by the coefficients ${\bsi}^{(r)}_{ij}$ of the series $\bsi_{ij}(u)=g_{ij}+\mysum_{r\geq1} \bsi^{(r)}_{ij} u^{-r}$ where $\bsi_{ij}(u)$ is the $(i,j)^{th}$-entry of the matrix $\bSi = \msq(u)^{-1} \msS(u)$. Observe that $\phi$ maps $\wt{\mcU\mcB}(\mcG)$ to $\wt Y(\mfg,\mcG)^{tw}$. It follows from Proposition \ref{ssw} that $\mysum_{a=-n}^n \bsi_{ia}(u)\,\bsi_{aj}(-u) = \delta_{ij}$.

The next theorem provides the analogue of Theorem \ref{T:Yt=Xt^Y} for reflection algebras.

\begin{thrm}\label{T:BGdecomp}
The extended twisted Yangian $X(\mfg,\mcG)^{tw}$ is isomorphic via $\phi$ to the algebra $\mcB(\mcG)$. The restriction of $\phi$ to $\wt{\mcU\mcB}(\mcG)$ provides an isomorphism between $\wt{\mcU\mcB}(\mcG)$ and $\wt Y(\mfg,\mcG)^{tw}$ such that $\phi: \bSi(u) \to \Si(u)$, $\bsi_{ij}(u)\mapsto \si_{ij}(u)$. Furthermore, ${\mcB}(\mcG)$ is isomorphic to $\mcW\mcB(\mcG)\ot\wt{\mcU\mcB}(\mcG)$ and the quotient homomorphism $\mcB(\mcG)\twoheadrightarrow{\mcU\mcB}(\mcG)$ induces an isomorphism between $\wt{\mcU\mcB}(\mcG)$ and ${\mcU\mcB}(\mcG)$. 
\end{thrm}

\begin{proof}
We start by showing that ${\mcB}(\mcG) \cong \mcW\mcB(\mcG)\ot\wt{\mcU\mcB}(\mcG)$ following the ideas in the proof of Theorem 3.1 in \cite{AMR}. If $f(u)$ is a power series in $u^{-1}$ such that $f(u) = f(\kappa-u)$, then the assignment $\msS(u) \mapsto f(u)\msS(u)$ defines an isomorphism of the reflection algebra ${\mcB}(\mcG)$. This isomorphism sends $\msw(u)$ to $f(u) f(-u) \msw(u)$. Because $\msw(u) = \msq(u) \msq(-u) = \msq(u) \msq(u+\kappa)$, it follows that it sends $\msq(u)$ to $f(u)\msq(u)$ and $\msq(u+\kappa)$ to $f(u+\kappa)\msq(u+\kappa)$, which equals $f(-u)\msq(u+\kappa)$. $\wt{\mcU\mcB}(\mcG)$ is thus invariant under such isomorphisms.

We have that ${\mcB}(\mcG) = \mcW\mcB(\mcG)\cdot \wt{\mcU\mcB}(\mcG)$. We need to show that the elements $\msw_{2i}$ are algebraically independent over $\wt{\mcU\mcB}(\mcG)$. Suppose, on the contrary, that $P(\msw_2, \msw_4, \ldots, \msw_{2n}) = 0$ for some polynomial $P$ in $n$ variables with coefficients in $\wt{\mcU\mcB}(\mcG)$. Since $\msw(u) = \msq(u) \msq(u+\kappa)$, it is enough to show that the elements $\msq_{2i}$ are algebraically independent over $\wt{\mcU\mcB}(\mcG)$ (notice that $\msq(-u) = \msq(u+\kappa)$ implies that the elements $\msq_j$ with $j$ odd can be expressed in terms of those with $j$ even), so we can consider instead a relation of the form $\wt{P}(\msq_2, \msq_4, \ldots, \msq_{2n}) = 0$ chosen so that $n$ is minimal. Let $f(u)=\left( 1+\frac{(-1)^n a u^{-n}}{(\kappa-u)^n}\right) = \left( 1+\frac{a u^{-2n}}{(1-\kappa u^{-1})^n}\right)$, so $f(u) = f(\kappa-u)$ and multiplication by $f(u)$ on $\msS(u)$ provides an isomorphism of ${\mcB}(\mcG)$. Since this isomorphism sends $\msq(u)$ to $f(u)\msq(u)$, it follows from $\wt{P}(\msq_2, \msq_4, \ldots, \msq_{2n}) = 0$ that $\wt{P}(\msq_2, \msq_4, \ldots, \msq_{2n}+ a) = 0$ for any $a\in\C$. Therefore, $\wt{P}$ does not depend on its last variable, which contradicts the choice of $n$. Consequently, the elements $\msw_{2i}$ are algebraically independent over $\wt{\mcU\mcB}(\mcG)$ and ${\mcB}(\mcG) \cong \mcW\mcB(\mcG)\ot\wt{\mcU\mcB}(\mcG)$.

Once the isomorphism ${\mcB}(\mcG) \cong \mcW\mcB(\mcG)\ot\wt{\mcU\mcB}(\mcG)$ has been established, it follows that the quotient homomorphism $\pi_1: \mcB(\mcG) \onto \mcU\mcB(\mcG)$ restricts to an isomorphism between $\wt{\mcU\mcB}(\mcG)$ and ${\mcU\mcB}(\mcG)$. Let $\pi_2: X(\mfg,\mcG)^{tw} \onto Y(\mfg,\mcG)^{tw}$ be the quotient homomorphism, so $\pi_2 \circ \phi = \hat{\phi} \circ \pi_1$. Since $\pi_2$ induces an isomorphism between $\wt Y(\mfg,\mcG)^{tw}$ and $Y(\mfg,\mcG)^{tw}$ by Theorem \ref{T:Yt=Xt^Y} and since $\phi$ maps $\wt{\mcU\mcB}(\mcG)$ to $\wt Y(\mfg,\mcG)^{tw}$, it follows from Theorem \ref{T:Y(g,G)} that $\phi$ restricts to an isomorphism between $\wt{\mcU\mcB}(\mcG)$ and $\wt Y(\mfg,\mcG)^{tw}$. Putting all this information together along with Theorem \ref{T:Yt=Xt^Y} allows us to conclude that $X(\mfg,\mcG)^{tw}$ is isomorphic via $\phi$ to $\mcB(\mcG)$.
\end{proof}

\begin{crl} \label{C:ZB(G)}
$\mcW\mcB(\mcG)$ is the whole centre $\mcZ\mcB(\mcG)$ of ${\mcB}(\mcG)$.
\end{crl} 
\begin{proof}
This is an immediate consequence of the previous theorem and Corollary \ref{C:ZX(g,G)}.
\end{proof}

\section{Connection with quantum contraction} \label{Sec:5}

Here we use an alternative approach of investigating the algebraic properties of the reflection algebras and twisted Yangians which was put forward in Section 6 in \cite{MNO}. This approach is based on the use of the one-dimensional projection operator $Q$. We construct certain series whose elements are central in the extended reflection algebra defined by the reflection equation only. The new constructed series are in one-to-one correspondence with the symmetry and unitarity relations.

In the computations below, we will use the following notation. Let $\{ e_i \}_{i=-n}^n$ denote the standard basis of $\C^N$. We have
$$
P\,(e_i \ot e_j ) = e_j \ot e_i , \qquad Q\,(e_i \ot e_j) = \delta_{-i,j} \mysum_{k} \theta_{jk}\,(e_{-k} \ot e_k) .
$$
These relations can be checked using the definitions \eqref{PQ}. We also set $\xi= \mysum_{k} \theta_{k1}\,(e_{-k} \ot e_k)$ so that $Q\,(\C^N\ot\C^N) = \C\,\xi$ and $Q\,(e_i \ot e_j) = \delta_{-i,j}\,\theta_{j1}\,\xi$.

\subsection{Extended reflection algebra} \label{Sec:51}

In this section, we define an extended reflection algebra $\mcX\mcB(\mcG)$ which depends on the $R$-matrix given by \eqref{R(u)} and the matrix $\mcG$ only. We then construct certain formal power series $\msc(u)$ in $u^{-1}$ with coefficients central in $\mcX\mcB(\mcG)$. This is an analogue of the series $\delta(u)$ constructed in Section 6 in \cite{MNO}. Then we show that the algebra $\mcB(\mcG)$ is isomorphic to the quotient of $\mcX\mcB(\mcG)$ by the ideal $\mcZ\mcX(\mcG)$ generated by the coefficients of the series $\msc(u)$, namely 
$$ \mcB(\mcG) \cong \mcX\mcB(\mcG) / ( \msc(u)-1 ), $$
or in other words, the constraint $\msc(u)=1$ is equivalent to the symmetry relation of $\mcB(\mcG)$. Moreover, we will show that the following tensor product decomposition holds:
$$
\mcX\mcB(\mcG) \cong \mcZ\mcX(\mcG) \ot \mcB(\mcG).
$$

\begin{defn}
The extended reflection algebra $\mcX\mcB(\mcG)$ is the unital associative $\C$-algebra generated by elements $\msx_{ij}^{(r)}$ for $-n\le i,j\le n, r\in\Z_{\ge 0}$ satisfying the reflection equation 
\eq{ \label{XRE}
R(u-v)\,\msX_1(u)\,R(u+v)\,\msX_2(v) = \msX_2(v)\,R(u+v)\,\msX_1(u)\,R(u-v) ,
}
where the $S$-matrix $\msX(u)$ is defined in the usual way: \[ \msX(u) = \mysum_{i,j=-n}^n \mysum_{r=0}^{\infty} E_{ij} \ot \msx_{ij}^{(r)} u^{-r} , \quad \msx_{ij}^{(0)} = g_{ij}. \]
\end{defn}

In what comes next, the following observation will be useful. Let $h(u)$ be a formal power series such that $h(u)\in1+u^{-1}\C[[u^{-1}]]$ and $A\in G$ a matrix such that $A\mcG A^t = \mcG$. The maps
\eq{
\wt\nu_h \;:\; \msX(u) \mapsto h(u)\,\msX(u) , \qquad \wt\gamma \;:\; \msX(u) \mapsto \msX^{-1}(-u) , \qquad \tl\al_A \;:\; \msX(u) \mapsto A \msX(u) A^t
}
are automorphisms of $\mcX\mcB(\mcG)$.

\begin{lemma} \label{L:QPRP} 
The matrix $\mcG(u)$ satisfies the following identity 
\eq{ \label{QPRP:1}
Q\,\mcG_1(u)\,R(2u-\ka)\,\mcG_2^{-1}(\ka-u) = \mcG_2^{-1}(\ka-u)\,R(2u-\ka)\,\mcG_1(u)\,Q = p(u)\,Q ,
}
where
\eq{
p(u)=(\pm)\,1\mp \dfrac{1}{2u-\ka}+\dfrac{\Tr(\mcG(u))}{2u-2\ka} \,. \label{p(u)}
}
\end{lemma}

\begin{proof}
Recall that the $R$-matrix $R(u)$ has a simple pole at $u=\ka$ with $\underset{u=\ka}{{\rm res}}R(u)=Q$. By multiplying both sides of \eqref{P-RE} with $\mcG^{-1}_2(v)$ we obtain
$$
R(u+v)\,\mcG_1(u)\,R(u-v)\,\mcG^{-1}_2(v) = \mcG^{-1}_2(v)\,R(u-v)\,\mcG_1(u)\,R(u+v) .
$$
Now multiply both sides of the previous equality by $u+v-\kappa$ and then set $v=-u+\kappa$. What remains is the first equality in \eqref{QPRP:1}. To prove the second equality, we need to consider each kind of $\mcG(u)$ individually.
Let $\mcG(u)$ of the first kind. In this case, the left-hand side of \eqref{QPRP:1} becomes
$$
Q\left((\pm)\,1\mp\frac{1}{2u-\ka}\right) + \frac{Q\,\mcG_1\,Q\,\mcG_2}{2u-2\ka} , 
$$
because $Q \mcG_1 P \mcG_2^{-1} = Q P = \pm Q $ and $Q \mcG_1 \mcG_2^{-1} = Q \mcG_1 \mcG_2 = Q \mcG_2^t \mcG_2$. By $QA_1Q=\Tr(A)\,Q$ and the properties of $\mcG$, it follows that
$$
Q\,\mcG_1\,Q\,\mcG_2 = \begin{cases} N Q = \Tr(\mcG(u))\,Q &\text{for the BCD0 case}, \\ 
0 &\text{for cases CI, DIII, DI and CII when }p=q. \end{cases} 
$$
Now let $\mcG(u)$ of the second kind. By $QA_1Q=\Tr(A)\,Q$ we have $Q \mcG_1 Q= \mysum_{i=-n}^n g_{ii}\,Q = (p-q)\,Q$. Recall that $c=4/(p-q)$, $\ka=N/2\mp1$ and $\Tr(\mcG(u))=(N-4u)(1-cu)^{-1}$. Then a straightforward (but tedious) calculation gives
\eqa{
Q & \left(\left(1\mp \frac{1}{2u-\ka}\right)\mcG_2(u)\,\mcG_2(u-\ka) + \frac{(N-4u)\,\mcG_2(u-\ka)}{(2u-2\ka)(1-cu)} \right) \el
& = Q \left(\left(1\mp \frac{1}{2u-\ka}\right) \frac{(1 + c^2 u (u-\ka)) I + c (\ka-2 u)\mcG}{(1-cu)(1-c(u-\ka))} + \frac{(N-4u)(I-c(u-\ka)\mcG)}{(2u-2\ka)(1-cu)(1-c(u-\ka))} \right) \el
& = Q\,\left(1 \mp \frac{1}{2u-\ka} + \frac{\Tr(\mcG(u))}{2u-2\ka} \right) . \nonumber
}
By combining the expressions above, we find $p(u)$ as given by \eqref{p(u)}.
\end{proof}

\begin{prop} \label{P:c(u)}
There exists a formal power series
$$ \msc(u) = 1 + \msc_1 u^{-1} + \msc_2 u^{-2} + \ldots \in \mcX\mcB(\mcG)[[u^{-1}]]$$
such that the following identity holds
\eq{ \label{QSRS:1}
Q\,\msX_1(u)\,R(2u-\ka)\,\msX_2^{-1}(\ka-u) = \msX_2^{-1}(\ka-u)\,R(2u-\ka)\,\msX_1(u)\,Q = p(u)\,\msc(u)\,Q .
}
\end{prop}

\begin{proof}
Multiply both sides of \eqref{XRE} by $\msX_2^{-1}(v)$ and $u+v-\kappa$ and then set $v=\kappa-u$. This gives the first equality in \eqref{QSRS:1}. The second equality follows from the fact that $Q/N$ is a projection operator to a one-dimensional subspace of $(\C^N)^{\ot2}$, thus the right-hand side must be equal to the operator $Q$ times some formal power series $\msc'(u)$ in $u^{-1}$ with coefficients in $\mcX\mcB(\mcG)$. The coefficient of $u^0$ in $\msc'(u)$ must be $(\pm)1$ since the coefficients of $u^0$ in the series $\msX_1(u)$, $\msX_2^{-1}(\ka-u)$ and $R(2u-\ka)$ are equal to $\mcG_1$, $\mcG_2^{-1}$ and $I$, respectively, giving $Q \mcG_1 \mcG_2^{-1}=Q \mcG^t_2 \mcG_2 = (\pm)\,Q$. Now since $p(u)$ given by \eqref{p(u)} is invertible, we can set $\msc(u) = p^{-1}(u)\,\msc'(u)$. This gives \eqref{QSRS:1} as required.
\end{proof}

\begin{rmk}
The identity \eqref{QSRS:1} together with \eqref{p(u)} are the analogues of those obtained in \cite{MNO}. In particular, by setting $\kappa=0$, choosing the plus sign in $(\pm)$ and discarding the last term in \eqref{p(u)} we recover the identity (1) in Proposition 6.2 in \textit{loc.\@ cit}. (see also Remark \ref{R:new-feat} in Section 4 above).
\end{rmk}

\begin{thrm} \label{T:c(u)} 
All the coefficients of the series $\msc(u)$ are central in $\mcX\mcB(\mcG)$. 
\end{thrm}

\noindent The proof of this theorem is similar to the one for Theorem 6.3 in \cite{MNO}.

\begin{proof}
Proving that $\msc(u)$ is central takes several steps. Consider the tensor space $(\End\,\C^N)^{\ot3}$. Enumerate the copies of $\End\,\C^N$ by $0,1,2$ and set 
$$
R_{ij}=R_{ij}(u_i-u_j), \quad R'_{ij}=R_{ij}(u_i+u_j), \quad \msX_i=\msX_i(u_i), \quad\text{with}\quad 0 \le i < j \le 2.
$$
We need to prove the identity 
\eq{
\msX_{0}\,\msc(u_1)\,Q_{12} = \msc(u_1)\,Q_{12}\,\msX_{0} , \label{T31:1}
}
which is equivalent to the statement that $\msc(u)$ is central. 
First, we need some auxiliary identities. Consider the following Yang-Baxter identities:
\eqa{ 
R_{12} R_{02} R_{01} & = R_{01} R_{02} R_{12}, \label{T31:3} \\
R'_{12} R'_{01} R_{02} &= R_{02} R'_{01} R'_{12} , \label{T31:4} \\
R'_{12} R'_{02} R_{01} &= R_{01} R'_{02} R'_{12} , \label{T31:5} \\
R_{12} R'_{01} R'_{02} &= R'_{02} R'_{01} R_{12} . \label{T31:6}
}
Here \eqref{T31:3} is the Yang-Baxter equation \eqref{YBE} written in the new notation. The remaining identities follow by transposing appropriate factors of the tensor space $(\End\,\C^N)^{\ot3}$ and using the property $R^t(u)=R(\ka-u)$. For example, to obtain \eqref{T31:4}, we need to transpose \eqref{T31:3} with $t_0$ and substitute $u_0\mapsto\ka-u_0$,\, $u_2\mapsto-u_2$. The remaining identities can be obtained in a similar same way. The reflection equation in the new notation reads as
\eq{ \label{T31:7}
R_{12} \msX_1 R'_{12} \msX_2 = \msX_2 R'_{12} \msX_1 R_{12}.
}
By multiplying both sides with $\msX^{-1}_2$ we get
\eq{ \label{T31:8}
R'_{12} \msX_1 R_{12} \msX^{-1}_2 = \msX^{-1}_2 R_{12} \msX_1 R'_{12} .
}
These auxiliary identities are needed to prove the following relation:
\eq{ \label{T31:9}
R_{01} R'_{02} \msX_0 R_{02} R'_{01} \msX^{-1}_2 R_{12} \msX_1 R'_{12} = \msX^{-1}_2 R_{12} \msX_1 R'_{12} R'_{01} R_{02} \msX_0 R'_{02} R_{01} .
}
Indeed,
\eqa{
R_{01} (R'_{02} \msX_0 R_{02} \msX^{-1}_2 ) R'_{01} R_{12} \msX_1 R'_{12}
& = R_{01} \msX^{-1}_2 R_{02} \msX_0 (R'_{02} R'_{01} R_{12}) \msX_1 R'_{12} \qquad\text{by \eqref{T31:8}} \el
&= R_{01} \msX^{-1}_2 R_{02} \msX_0 R_{12} R'_{01} R'_{02} \msX_1 R'_{12} \phantom{()} \qquad\text{by \eqref{T31:6}} \el
&= \msX^{-1}_2 (R_{01} R_{02} R_{12}) \msX_0 R'_{01} R'_{02} \msX_1 R'_{12} \el 
&= \msX^{-1}_2 R_{12} R_{02} ( R_{01} \msX_0 R'_{01} \msX_1) R'_{02} R'_{12} \qquad\text{by \eqref{T31:3}} \el
&= \msX^{-1}_2 R_{12} R_{02} \msX_1 R'_{01} \msX_0 (R_{01} R'_{02} R'_{12}) \qquad\text{by \eqref{T31:7}} \el
&= \msX^{-1}_2 R_{12} \msX_1 (R_{02} R'_{01} R'_{12}) \msX_0 R'_{02} R_{01} \qquad\text{by \eqref{T31:5}} \el
&= \msX^{-1}_2 R_{12} \msX_1 R'_{12} R'_{01} R_{02} \msX_0 R'_{02} R_{01} \phantom{()} \qquad\text{by \eqref{T31:4}} , \nonumber
}
thus proving \eqref{T31:9}. Now multiply both sides of \eqref{T31:9} by $u_1+u_2-\kappa$ and set $u_2 = \kappa-u_1$. By Proposition \ref{P:c(u)} we obtain
\eq{
R_{01} R'_{02} \msX_0 R_{02} R'_{01} Q_{12}\,\msc(u_1) = \msc(u_1)\,Q_{12} R'_{01} R_{02} \msX_0 R'_{02} R_{01} \label{T31:1A} .
}
We will use the following identities to simplify \eqref{T31:1A}:
\eqa{
 Q_{12} R'_{01} R_{02} &= R_{02} R'_{01} Q_{12} = (1 - (u_0+u_1-\ka)^{-2}) \, Q_{12}\, , \label{T31:10}\\
 Q_{12} R'_{02} R_{01} &= R_{01} R'_{02} Q_{12} = (1 - (u_0-u_1)^{-2}) \, Q_{12}\, , \label{T31:11}
}
which follow from \eqref{T31:6} and \eqref{T31:3}, respectively, after replacing $u_2$ by $-u_2$ and taking the residue at $u_1+u_2=\kappa$. Let us explicitly show how to obtain \eqref{T31:10}. Since $Q/N$ is a one-dimensional projector it is sufficient to consider the action of $Q_{12} R'_{01} R_{02}$ on the basis vector $\eta = e_i\ot e_{-1} \ot e_1\in(\C^N)^{\ot3}$, since $Q_{12}\,\eta = e_i \ot \xi$. Define $u_{ij}=u_i-u_j$ and $v_{ij}=u_i+u_j$. This gives
\eqn{
& Q_{12} R'_{01} R_{02} (e_i\ot e_{-1} \ot e_1) \\
& \qquad = Q_{12} R'_{01} \left( e_i\ot e_{-1} \ot e_1 - \frac{1}{u_{02}}\,e_1\ot e_{-1} \ot e_i + \frac{\delta_{i,-1}}{u_{02}-\kappa}\mysum_{j}\theta_{1j}\,e_{-j}\ot e_{-1} \ot e_j \right) ,
}
which evaluates to
\eqa{
& Q_{12} ( e_i\ot e_{-1} \ot e_1 - \frac{1}{u_{02}}\,e_1\ot e_{-1} \ot e_i + \frac{\delta_{i,-1}}{u_{02}-\kappa}\mysum_{j}\theta_{1j}\,e_{-j}\ot e_{-1} \ot e_j \el
& \qquad -\dfrac{1}{v_{01}}\, e_{-1}\ot e_i \ot e_{1} + \dfrac{1}{v_{01}u_{02}}\,e_{-1}\ot e_1 \ot e_{i} - \dfrac{\delta_{i,-1}}{v_{01}(u_{02}-\kappa)}\mysum_{j}\theta_{1j}\, e_{-1}\ot e_{-j} \ot e_j \el
& \qquad \pm \frac{\delta_{i1}}{v_{01}-\kappa}\mysum_{j}\theta_{1j}\,e_{-j} \ot e_j \ot e_{1}
 \mp \dfrac{1}{u_{02}(v_{01}-\kappa)}\mysum_{j}\theta_{1j} \,e_{-j}\ot e_j \ot e_i \el
& \hspace{3in} 
 + \dfrac{\delta_{i,-1}}{(v_{01}-\kappa)(u_{02}-\kappa)}\mysum_{j}\theta_{1j}\,e_{-j}\ot e_{j} \ot e_{-1} ) \nn\\
& = \mysum_j\theta_{1j}\, e_i\ot e_{-j} \ot e_j - \dfrac{\delta_{i1}}{u_{02}}\mysum_j \theta_{1j}\,e_1\ot e_{-j} \ot e_j + \dfrac{\delta_{i,-1}}{u_{02}-\kappa}\mysum_{j}\theta_{1j}\,e_{-1}\ot e_{-j} \ot e_j \el
& \qquad - \frac{\delta_{i,-1}}{v_{01}}\mysum_j \theta_{1j}\, e_{-1}\ot e_{-j} \ot e_{j} \pm \dfrac{\delta_{i,-1}}{v_{01}u_{02}}\mysum_j \theta_{1j}\,e_{-1}\ot e_{-j} \ot e_{j} - \dfrac{N\delta_{i,-1}}{v_{01}(u_{02}-\kappa)}\mysum_{j}\theta_{1j}\, e_{-1}\ot e_{-j} \ot e_j \el
& \qquad + \dfrac{\delta_{i1}}{v_{01}-\kappa}\mysum_{j}\theta_{1j}\,e_{1} \ot e_{-j} \ot e_{j}
 - \dfrac{1}{u_{02}(v_{01}-\kappa)}\mysum_{j}\theta_{1j} \,e_{i}\ot e_{-j} \ot e_j \el
& \hspace{3in} 
 \pm \frac{\delta_{i,-1}}{(v_{01}-\kappa)(u_{02}-\kappa)}{\mysum_{j}}\theta_{1j}\,e_{-1}\ot e_{-j} \ot e_{j} .
} \notag
After substituting $u_2\to\ka-u_1$ most of the terms cancel each other. What remains is
$$
(1 - (u_0+u_1-\ka)^{-2}) \mysum_j\theta_{1j}\, e_i\ot e_{-j} \ot e_j = (1 - (u_0+u_1-\ka)^{-2}) ( e_i\ot \xi ) ,
$$
which implies \eqref{T31:10}. A similar calculation for $Q_{12} R'_{02} R_{01}$ implies \eqref{T31:11}. These two relations applied to \eqref{T31:1A} give \eqref{T31:1}. This proves the theorem.
\end{proof}

\begin{crl} \label{C:52}
The odd coefficients $\msc_1, \msc_3, \ldots$ of the series $\msc(u)$ are algebraically independent.
\end{crl}

\begin{proof}
Consider the polynomial ring $\C[x_1,x_2,\ldots]$ in infinitely many variables and set $f(u) = 1 + \mysum_{r=1}^{\infty} x_r u^{-r}$. We have that $f(u) \mcG(u)$ is a solution of the reflection equation by Lemma \ref{L:RPRP}. It follows that the assignment $\msX(u) \mapsto f(u)\mcG(u)$ defines an algebra homomorphism $\beta_f:\End(\C^N) \otimes \mcX\mcB(\mcG)[[u^{-1}]] \lra \End(\C^N) \otimes \C[x_1,x_2,\ldots][[u^{-1}]]$. Applying $\beta_f$ to both sides of \eqref{QSRS:1}, we obtain that $f(u) f(\kappa - u)^{-1} = \beta_f(\msc(u))$ by Lemma \ref{L:QPRP}. Therefore, $\beta_f(\msc_{2r+1}) = 2 x_{2r+1} + g_{2r}$ where $g_{2r}$ is a polynomial in the variables $x_1,\ldots, x_{2r}$. Since the variables $x_i, \, i\ge 1$ are algebraically independent, so are $\beta_f(\msc_{2r+1}) \; \forall \, r\ge 0$, and the same must be true for the central elements $\msc_{2r+1}$ for all $r\ge 0$.
\end{proof}

\begin{lemma} \label{L:QSRS} 
The $S$-matrix $S(u)$ given in Definition \ref{D:X(g,G)} satisfies the symmetry relation 
\eq{ \label{QSRS:1A}
Q\,S_1(u)\,R(2u-\ka)\,S_2^{-1}(\ka-u) = S_2^{-1}(\ka-u)\,R(2u-\ka)\,S_1(u)\,Q = p(u)\,Q .
} where $p(u)$ is the power series given in \eqref{p(u)}.
\end{lemma}
\begin{proof}
The proof of the first equality is analogous to the one in the proof of the Proposition \ref{P:c(u)} above. Proving the second equality requires the following auxiliary relation:
\eqa{ 
T^t_1(-u+\ka/2)\,R(2u-\ka)\,T^t_2(u-\ka/2)^{-1} &= T^t_2(u-\ka/2)^{-1}\,R(2u-\ka)\,T^t_1(-u+\ka/2) , \label{L51:1} 
}
which is obtained by multiplying both sides of \eqref{L32:3} with $T^t_2(-v+\ka/2)^{-1}$ and substituting $v\mapsto-u+\ka$. Now recall that $Q\,T_1(u)=Q\,T^t_2(u)$ and $\mcG^{-1}(u)=\mcG(-u)$. We have
\eqn{
Q\,&S_1(u)\,R(2u-\ka)\,S_2^{-1}(\ka-u) \\
&=Q\,T_1(u-\ka/2)\,\mcG_1(u)\,(T^t_1(-u+\ka/2)\,R(2u-\ka)\,T^t_2(u-\ka/2)^{-1})\,\mcG_2(u-\ka)\,T^{-1}_2(-u+\ka/2) \\
& = Q\,T_1(u-\ka/2)\,T^t_2(u-\ka/2)^{-1}\,\mcG_1(u)\,R(2u-\ka)\,T^t_1(-u+\ka/2)\,\mcG_2(u-\ka)\,T^{-1}_2(-u+\ka/2) \quad\text{by \eqref{L51:1}} \\
& = Q\,\mcG_1(u)\,R(2u-\ka)\,\mcG_2(u-\ka)\,T^t_1(-u+\ka/2)\,T^{-1}_2(-u+\ka/2) \\
& = p(u)\,Q\,\,T^t_1(-u+\ka/2)\,T^{-1}_2(-u+\ka/2) = p(u)\,Q \hspace{6.6cm}\text{by } \eqref{QPRP:1}. 
}
\end{proof}

We have the following equivalence:

\begin{thrm}\label{T:XRES}
The relation $\msc(u)=1$ is equivalent to the symmetry relation
\eq{
\msX^t(u) = (\pm)\,\msX(\ka-u) \pm \frac{\msX(u)-\msX(\ka-u)}{2u-\ka} + \frac{\Tr(\mcG(u))\,\msX(\kappa-u) - \Tr(\msX(u))\cdot I}{2u-2\ka} . \label{XRES}
}
\end{thrm}

\begin{proof}
Denote the matrix elements of $\msX^{-1}(u)$ by $\msx'_{ij}(u)$, $-n\leq i,j\leq n$, and apply the left-hand side of \eqref{QSRS:1} to the vector $e_{-i}\ot e_{j}\in (\C^N)^{\ot2}$. This gives
\eqa{
Q&\,\msX_1(u)\,R(2u-\ka)\,\msX_2^{-1}(\ka-u) \, (e_{-i}\ot e_{j}) = Q\,\msX_1(u)\,R(2u-\ka) \mysum_{k=-n}^n \msx'_{kj}(\ka-u)(e_{-i}\ot e_{k}) \el
& = Q\,\msX_1(u) \mysuml_{k=-n}^n \left( \msx'_{kj}(\ka-u) \left( e_{-i}\ot e_k - \dfrac{1}{2u-\ka} \,e_{k}\ot e_{-i}\right) + \dfrac{\theta_{ik}\,\msx'_{ij}(\ka-u)}{2u-2\ka} \, e_{-k}\ot e_k \right) \el
& = Q \mysuml_{k,l} \left( \left(\msx_{l,-i}(u)\,\msx'_{kj}(\ka-u)\, e_{l}\ot e_k - \dfrac{\msx_{lk}(u)\,\msx'_{kj}(\ka-u)}{2u-\ka}\, e_{l}\ot e_{-i}\right) + \dfrac{\theta_{ik}\,\msx_{l,-k}(u)\,\msx'_{ij}(\ka-u)}{2u-2\ka}\, e_{l}\ot e_k \right) \el
& = \mysuml_{k} \left( \theta_{k1}\,\msx_{-k,-i}(u)\,\msx'_{kj}(\ka-u) - \dfrac{\theta_{-i,1}\,\msx_{ik}(u)\,\msx'_{kj}(\ka-u)}{2u-\ka} + \dfrac{\theta_{i1}\,\msx_{-k,-k}(u)\,\msx'_{ij}(\ka-u)}{2u-2\ka} \right) \xi\, . \label{pcQ1}
}
For the right-hand side of \eqref{QSRS:1} we have 
\eqa{
p(u)\,\msc(u)\, Q\,(e_{-i}\ot e_j) = p(u)\,\msc(u)\,\delta_{ij}\,\theta_{i1}\,\xi . \label{pcQ2}
}
Recall that $\delta_{ij} = \mysum_k \msx_{ik}(\ka-u)\,\msx'_{kj}(\ka-u)$ and set $\msx'_{ij}(\ka-u)=\mysum_k \delta_{ik}\,\msx'_{kj}(\ka-u)$. Then by comparing the equalities \eqref{pcQ1} and \eqref{pcQ2} above we find
\eq{
p(u)\,\msc(u)\,\msx_{ik}(\ka-u) = \theta_{ki}\,\msx_{-k,-i}(u) \mp \frac{\msx_{ik}(u)}{2u-\ka} + \frac{\delta_{ik}\mysum_l\msx_{ll}(u)}{2u-2\ka} . \label{T52:1}
}
By setting $\msc(u)=1$ and using \eqref{p(u)}, the explicit form of $p(u)$, we obtain the relation
$$
\theta_{ki}\,\msx_{-k,-i}(u) = (\pm)\,\msx_{ik}(\ka-u) \mp \dfrac{\msx_{ik}(\ka-u)-\msx_{ik}(u)}{2u-\ka}+\dfrac{\Tr(\mcG(u))\,\msx_{ik}(\ka-u) - \delta_{ik} \mysum_l\msx_{ll}(u)}{2u-2\ka}
$$
which is equivalent to a matrix element of the symmetry relation \eqref{XRES}. On the other hand, if \eqref{XRES} is satisfied, then \eqref{T52:1} gives, for all $i$ and $k$,
$$
p(u)\,\msc(u)\,\msx_{ik}(\ka-u) = \theta_{ki}\,\msx_{-k,-i}(u) \mp \frac{\msx_{ik}(u)}{2u-\ka} + \frac{\delta_{ik}\mysum_l\msx_{ll}(u)}{2u-2\ka} = p(u)\,\msx_{ik}(\ka-u) ,
$$
where we have used \eqref{XRES} to obtain the second equality. Hence $p(u)\,(\msc(u)-1)\,\msX(\ka-u)=0$. The series $p(u)$ is invertible, and so is the matrix $\msX(\ka-u)$, since its constant term is $\mcG$. So $\msc(u)=1$. 
\end{proof}

\begin{crl}\label{C:XB->B}
The reflection algebra $\mcB(\mcG)$ is isomorphic to the quotient of $\mcX\mcB(\mcG)$ by the ideal generated by the coefficients of the series $\msc(u)-1$:
\eq{
\mcB(\mcG) \cong \mcX\mcB(\mcG) / (\msc(u) - 1) . \nn
}
\end{crl}

\begin{prop} \label{P:S-inv}
The algebra $\mcB(\mcG)$ is invariant under the automorphism $\wt\nu_h$ for any series $h(u)$ satisfying $h(u)\,h^{-1}(\ka-u)=1$.
\end{prop}

\begin{proof}
By \eqref{QSRS:1} the image of $\msc(u)$ under the automorphism $\wt\nu_h$ is $h(u)\,h^{-1}(\ka-u)\,\msc(u)$. If $h(u) = h(\kappa-u)$, then $\wt{\nu}_h(\msc(u) - 1) = \msc(u) -1$, so the ideal generated by the coefficients of $\msc(u)-1$ is stable under $\wt{\nu}_h$ and this automorphism descends to the quotient $\mcX\mcB(\mcG)/(\msc(u)-1)$.
\end{proof}

From \eqref{QSRS:1}, we quickly obtain 
\[ 
\msc(u)^{-1} Q = p(u)\, Q\, \msX_2(\kappa-u) \, R^{-1}(2u-\kappa) \msX_1^{-1}(u) 
\] 
which, after conjugating by $P$, gives 
\[ 
\msc^{-1}(u)\, Q = p(u)\, Q\, \msX_1(\kappa-u)\, \frac{R(\kappa-2u)}{1-(2u-\kappa)^{-2}}\, \msX_2^{-1}(u) = \frac{p(u)\,p(\kappa-u)}{1-(2u-\kappa)^{-2}}\,\msc(\kappa-u)\, Q. 
\] 
It can be checked that $p(u)\,p(\kappa-u) = 1 - (2u-\kappa)^{-2}$, so it follows that $\msc(\kappa-u) = \msc^{-1}(u)$.

Let $\msv(u)$ be the unique invertible power series with constant term $1$ such that $\msc(u)=\msv^2(u)$. Then $\msv^{-2}(u) = \msc(\kappa-u) = \msv^2(\kappa-u)$, hence $\msv^{-1}(u) = \msv(\kappa-u)$ because both have constant term 1. Thus it follows that $\msc(u) = \msv(u)\,\msv^{-1}(\ka-u)$. 

For any power series $h(u)$, we have $\wt\nu_h(\msc(u))=h(u)\,h^{-1}(\ka-u) \,\msc(u)$ (see \eqref{QSRS:1}). If $h(u)$ satisfies $h^{-1}(u) = h(\kappa-u)$ with constant term 1, then we deduce that $\wt\nu_h(\msc(u))=h^2(u) \,\msc(u)$ and $\tl\nu_h(\msv(u))=h(u)\,\msv(u)$.

\begin{thrm} \label{T:XS=ZX*S}
Let $\wt{\mcB}(\mcG)$ be the subalgebra of $\mcX\mcB(\mcG)$ generated by the coefficients of the series $\tl\mss_{ij}(u)=\msv^{-1}(u)\,\msx_{ij}(u)$. $\mcX\mcB(\mcG)$ is isomorphic to $\mcZ\mcX(\mcG) \ot \mcB(\mcG)$. Moreover, the quotient homomorphism $\mcX\mcB(\mcG)\onto \mcB(\mcG)$ induces an isomorphism between $\wt{\mcB}(\mcG)$ and $\mcB(\mcG)$. 
\end{thrm}

\begin{proof}
Since the coefficients of the series $\msc(u)$ generate $\mcZ\mcX(\mcG)$, the same is true for the coefficients of $\msv(u)$. Consequently, since $\msx_{ij}(u)=\msv(u)\,\tl\mss_{ij}(u)$, it follows that $\mcX\mcB(\mcG)\cong \mcZ\mcX(\mcG)\cdot\wt\mcB(\mcG)$. Moreover if $h^{-1}(u) = h(\kappa-u)$ with constant term 1, the algebra $\wt{\mcB}(\mcG)$ is a $\wt\nu_h$--stable subalgebra of $\mcX\mcB(\mcG)$. Indeed, $\wt\nu_h(\tl\mss_{ij}(u))=\wt\nu_h(\msv^{-1}(u))\,\wt\nu_h(\msx_{ij}(u))=h^{-1}(u)\,\msv^{-1}(u)\,h(u)\,\msx_{ij}(u)=\tl\mss_{ij}(u)$. The isomorphism $\mcX\mcB(\mcG) \cong \mcZ\mcX(\mcG) \otimes \wt{\mcB}(\mcG)$ can now be proved via the same argument as in Theorem 3.1 in \cite{AMR} using instead of $1+au^{-n}$ the power series $(1+a(\kappa-u)^{-n})/(1+au^{-n})$ for an appropriate odd value of $n$ and any $a\in\C$. (When $n$ is odd, it is important that the first two terms of this power series are $1 - 2au^{-n}$: 
\[ 
\frac{1+a(\kappa-u)^{-n}}{1+au^{-n}} = \left( 1-\frac{au^{-n}}{(1-\kappa u^{-1})^{n}}\right) (1-au^{-n} + a^2 u^{-2n} - \cdots) = 1-2au^{-n} + \cdots 
\] 
The reason why $n$ should be odd is that the odd coefficients of $\msv(u)$ should be considered.) It follows that the quotient $\mcX\mcB(\mcG) \onto \mcB(\mcG)$ induces an isomorphism between $\wt{\mcB}(\mcG)$ and $\mcB(\mcG)$. 
\end{proof}

\begin{crl} \label{XX:PBW} 
Given any total ordering, a vector space basis of $\mcX\mcB(\mcG)$ is provided by the ordered monomials in the generators $\msc_1, \msc_3, \ldots$ and $\msw_2, \msw_4, \ldots$ and $\bsi^{(r)}_{ij}$ with $r,i,j$ satisfying the same constraints as in Theorem~\ref{Y:PBW}.
\end{crl}

\subsection{Quantum contraction for reflection algebra} \label{Sec:52}

In this section, we define a certain series $\msd(u)$, the image of $\wt\gamma(\msc(u))$ in the algebra $\mcB(\mcG)$. We call this series the quantum contraction of the matrix $\msS(u)$ in an analogy to the quantum contraction $y(u)$ of the twisted Yangian in \cite{MNO}. We then show that $\msd(u)$ is an analogue of the series $\msw(u)$.

\begin{prop}
The following identity holds in the algebra $\mcB(\mcG)$:
\eq{
Q\,\msS^{-1}_1(-u)\,R(2u-\ka)\,\msS_2(u-\ka) = \msS_2(u-\ka)\,R(2u-\ka)\,\msS^{-1}_1(-u)\,Q = p(u)\,\msd(u)\,Q . \label{QSRS:2}
}
\end{prop}

\begin{proof}
Apply the automorphism $\wt\gamma$ to each part of \eqref{QSRS:1} and take their image in the algebra $\mcB(\mcG)$. 
\end{proof}

\begin{thrm} \label{T:d(u)}
The coefficients of the quantum contraction $\msd(u)$ generate the whole centre $\mcZ\mcB(\mcG)$ of $\mcB(\mcG)$.
\end{thrm}

\begin{proof} 
Set $d(u) = \phi(\msd(u))$ and let us apply the isomorphism $\phi: \mcB(\mcG) \to X(\mfg,\mcG)^{tw}$ to the left-hand side of \eqref{QSRS:2} to obtain
\eqa{
& Q\,T^t_1(u+\ka/2)^{-1}\,\mcG_1(u)\,T^{-1}_1(-u-\ka/2)\,R(2u-\ka)\,T_2(u-3\ka/2)\,\mcG_2(u-\ka)\,T_2^{t}(-u+3\ka/2) \el
& = Q\,T^t_1(u+\ka/2)^{-1}\,T_2(u-3\ka/2)\,\mcG_1(u)\,R(2u-\ka)\,\mcG_2(u-\ka)\,T^{-1}_1(-u-\ka/2)\,T_2^{t}(-u+3\ka/2) , \label{P36:1}
}
where we have used $\mcG^{-1}(-u)=\mcG(u)$ and the identity
$$
T^{-1}_1(-u-\ka/2)\,R(2u-\ka)\,T_2(u-3\ka/2) = T_2(u-3\ka/2)\,R(2u-\ka)\,T^{-1}_1(-u-\ka/2) ,
$$
which is obtained by taking the inverse of \eqref{RTT}, multiplying both sides with $T_2(v)$ and substituting $u\to -u-\ka/2$, $v\to u-3\ka/2$. Now recall that $Q\,T_2^{t}(u)=Q\,T_1(u)$. Then, by \eqref{TT}, it follows that \eqref{P36:1} is equal to
$$
p(u)\,\frac{y(u-3\ka/2)\,y(-u+3\ka/2)}{y(u+\ka/2)\,y(-u-\ka/2)}\,Q = p(u)\,\frac{q(u-\ka)}{q(-u)}\,Q ,
$$
which yields, after comparing with the right-hand side of \eqref{QSRS:2},
\eq{
d(u) = \frac{q(u-\ka)}{q(-u)} = \frac{w(\ka-u)}{w(u)}  . \label{c=q/q}
}
The second equality holds because $w(u)=q(u)\,q(-u)$ and $q(u)=q(\ka-u)$, so that $w(\ka-u)=q(\ka-u)\,q(u-\ka)=q(u)\,q(u-\ka)$.

Now the coefficients of $q(u)$ generate the whole centre $ZX(\mfg,\mcG)^{tw}$. Set $v=u-\ka/2$. Then $\hat{q}(v)=q(v+\ka/2)=y(v)\,y(-v)$ is an even series in $v^{-1}$; write $L(v)=\log\hat{q}(v)=\mysum_{k\ge1}l_{2k}\,v^{-2k}$. Then
$$
\log d(u) = L(v-\ka)-L(v+\ka) = \mysum_{k\ge1} 4k\,\ka\, l_{2k}\, v^{-2k-1} + \cdots ,
$$
and, since $\ka\neq0$, this is triangular in the coefficients $l_{2k}$. Hence the same is true for the series $d(u)$ and, by the isomorphism $\phi$, for $\msd(u)$.
\end{proof}

\begin{crl}
We have $\mcX\mcB(\mcG) \cong \mcZ\mcX(\mcG) \ot \mcZ\mcB(\mcG) \ot \mcU\mcB(\mcG)$ and $\mcZ\mcX(\mcG) \ot \mcZ\mcB(\mcG)$ is the centre of $\mcX\mcB(\mcG)$.
\end{crl}

\begin{proof}
This follows from Theorem \ref{T:XS=ZX*S} and Theorem \ref{T:BGdecomp}.
\end{proof}



\end{document}